\documentclass[11pt]{article}

\usepackage[table]{xcolor}
\usepackage{IEEEtrantools}
\usepackage{latexsym}
\usepackage{amsfonts}
\usepackage{amssymb}
\usepackage{amsmath,amsfonts,amsthm}
\usepackage{mathrsfs}
\usepackage[all]{xy}
\usepackage{epic}
\usepackage{eepic}
\usepackage{enumerate}
\usepackage{multirow}
\usepackage[breaklinks]{hyperref}
\usepackage{tikz}

\usepackage{mathrsfs}
\usepackage[all]{xy}
\usepackage{epic}
\usepackage{eepic}
\usepackage{enumerate}
\usepackage{multirow}
\usepackage[breaklinks]{hyperref}
\usepackage{tikz}
\usetikzlibrary{matrix,arrows,decorations.pathmorphing}
\usepackage{xcolor}
\usepackage{color}
\usepackage[breaklinks]{hyperref}
\usepackage{tikz-cd}

\usepackage{mathrsfs}
\usepackage[all]{xy}
\usepackage{epic}
\usepackage{eepic}
\usepackage{enumerate}
\usepackage{multirow}
\usepackage[breaklinks]{hyperref}
\usepackage{tikz}
\usetikzlibrary{arrows,decorations.markings}
\usepackage{color}
\usepackage{bbm}
\usepackage[most]{tcolorbox}
\usepackage{subcaption}
\usetikzlibrary{patterns}

\newcommand{\bbbt}{\mathbb{T}}
\newcommand{\scrt}{\mathscr{T}}

\newcommand{\be}{\begin{equation}}
	\newcommand{\ee}{\end{equation}}
\newcommand{\bea}{\begin{eqnarray}}
	\newcommand{\eea}{\end{eqnarray}}
\newcommand{\bean}{\begin{eqnarray*}}
	\newcommand{\eean}{\end{eqnarray*}}
\newcommand{\brray}{\begin{array}}
	\newcommand{\erray}{\end{array}}
\newcommand{\biearray}{\begin{IEEEarray}{rCl}}
	\newcommand{\eiearray}{\end{IEEEarray}}

\newcommand{\newsection}[1]{\setcounter{equation}{0}
	\setcounter{dfn}{0}
	\section{#1}}

\newtheorem{dfn}{Definition}[section]
\newtheorem{thm}[dfn]{Theorem}
\newtheorem{lmma}[dfn]{Lemma}
\newtheorem{ppsn}[dfn]{Proposition}
\newtheorem{crlre}[dfn]{Corollary}
\newtheorem{xmpl}[dfn]{Example}
\newtheorem{rmrk}[dfn]{Remark}
\newtheoremstyle{case}{}{}{}{}{}{:}{ }{}
\theoremstyle{case}

\newtheoremstyle{claim}{}{}{}{}{}{:}{ }{}
\theoremstyle{claim}
\newtheorem{claim}{Claim}

\newcommand{\bdfn}{\begin{dfn}\rm}
	\newcommand{\bthm}{\begin{thm}}
		\newcommand{\blmma}{\begin{lmma}}
			\newcommand{\bppsn}{\begin{ppsn}}
				\newcommand{\bcrlre}{\begin{crlre}}
					\newcommand{\bxmpl}{\begin{xmpl}}
						\newcommand{\brmrk}{\begin{rmrk}\rm}
							
							\newcommand{\edfn}{\end{dfn}}
						\newcommand{\ethm}{\end{thm}}
					\newcommand{\elmma}{\end{lmma}}
				\newcommand{\eppsn}{\end{ppsn}}
			\newcommand{\ecrlre}{\end{crlre}}
		\newcommand{\exmpl}{\end{xmpl}}
	\newcommand{\ermrk}{\end{rmrk}}

\newcommand{\bbc}{\mathbb{C}}
\newcommand{\bbz}{\mathbb{Z}}
\newcommand{\bbn}{\mathbb{N}}

\newcommand{\cla}{\mathcal{A}}

\newcommand{\clh}{\mathcal{H}}

\newcommand{\clk}{\mathcal{K}}
\newcommand{\cll}{\mathcal{L}}

\def \bbt {\mbox{\boldmath $t$}}

\newcommand{\prf}{\noindent{\it Proof}. }

\def \qed { \mbox{}\hfill
	$\Box$\vspace{1ex}}

\newcommand\restr[2]{{
		\left.\kern-\nulldelimiterspace 
		#1 
		\littletaller 
		\right|_{#2} 
}}

\newcommand{\littletaller}{\mathchoice{\vphantom{\big|}}{}{}{}}

\newcommand\scalemath[2]{\scalebox{#1}{\mbox{\ensuremath{\displaystyle #2}}}}
\newenvironment{proof*}[1][\proofname]{\proof[#1]}{\endproof}

\begin{document}

	\author{{\sc Akshay Bhuva, Surajit Biswas, Bipul Saurabh}}
	\date{}
	\title{ Topological invariance of quantum homogeneous spaces of type $B$ and $D$ }
	\maketitle
	\begin{abstract}

		In this article,  we study two families of quantum homogeneous spaces, namely,  $SO_q(2n+1)/SO_q(2n-1)$, and $SO_q(2n)/SO_q(2n-2)$. By applying a two-step Zhelobenko branching rule, we show that the $C^*$-algebras $C(SO_q(2n+1)/SO_q(2n-1))$, and $C(SO_q(2n)/SO_q(2n-2))$ are generated by the entries of the first and the last rows of the fundamental matrix of the quantum groups $SO_q(2n+1)$, and $SO_q(2n)$, respectively.  We then construct  a chain of short exact sequences, and using that, we  compute $K$-groups of these spaces with explicit generators. 
		Invoking  homogeneous $C^*$-extension theory, we show $q$-independence of some  intermediate $C^*$-algebras arising as the middle $C^*$-algebra of these  short exact sequences. As a consequence, we get the $q$-invariance of $SO_q(3)$, $SO_q(5)/SO_q(3)$,  $SO_q(4)/SO_q(2)$, and $SO_q(6)/SO_q(4)$.

	\end{abstract}

	{\bf AMS Subject Classification No.:} {\large 58}B{\large 34}, {\large
		46}L{\large 80}, {\large
		19}K{\large 33}

	\section{Introduction}
	
	Let $G$ be a semisimple compact  Lie group with complexified Lie algebra $\mathfrak{g}$.  Fix $ 0<q<1$. 
	The algebra of functions $C(G_q)$ on its  $q$-deformation $G_q$ is defined as the  
	enveloping $C^*$-algebra of the Hopf $*$-algebra generated by 
	matrix coefficients of all finite-dimensional representations of Quantized universal enveloping algebra (QUEA)  $U_q(\mathfrak{g})$.  It turns out that if $H$ is a closed  Poisson Lie subgroup of $G$, then  $H_q$ is a quantum subgroup of $G_q$ and the $C^*$-algebra 
	$C(G_q/H_q)$ underlying the  quotient space $G_q/H_q$
	is a $C^*$-subalgebra of $C(G_q)$ generated by matrix elements of certain finite dimensional representations
	of $U_q(\mathfrak{g})$.   One of the main objectives in noncommutative geometry (NCG) is to see how the theory of quantum groups and their quotient spaces fits under Connes formulation of NCG (see \cite{Con-1994aa} for details).  To achieve this, it is crucial to examine the $C^*$-algebra underlying these spaces.  The direct approach of exploring the operators obtained as images of the generators of $C(G_q/H_q)$ under a faithful representation, as mentioned in (\cite{KorSoi-1998aa}, \cite{NesTus-2012ab}), seems to be complicated. Other possible approaches can be to see whether the given $C^*$- algebra can be associated with a graph, groupoid, or semigroup or it can be obtained by applying some noncommutative operations on  a classical space. Many articles have investigated this (see \cite{HonSzy-2002aa}, \cite{She-1997ac}). However, most discuss $SU_q(n)$ and its homogeneous spaces. The $C^*$-algebra of quotient spaces of other types remains less explored.   To give a glimpse of the situation, the very first question of proving whether $C(SO_q(\mathcal{N})/SO_q(\mathcal{N}-2))$ is given by a finite set of generators and relations  has yet to be answered. Moreover, since $SO(n)$ is not simply connected for $n \geq 3$, one cannot simply apply the general theory developed in \cite{NesTus-2012ab}, which  adds another layer of difficulty to the situation.  Lance (\cite{Lan-1998aa}) and Podles (\cite{Pod-1995aa}) have done some explicit computations, but  only for $C(SO_q(3))$.  Recent works  (see \cite{Sau-2017aa}, \cite{Sau-2019aa}, \cite{ChaSau-2018ab}, \cite{ChaSau-2019aa},  \cite{BhuSau-2023aa}) have made some progress but have only scratched the surface of these challenges.  Therefore, it is worthwhile to investigate such spaces. 
	In this article, we take up certain homogeneous spaces of quantized special orthogonal groups and explore their topological properties using the tools of extension theory. 
	
	$C^*$-extension theory has its origin in the work of 
	Brown, Douglas, and Fillmore (\cite{BroDouFil-1977aa}), where the authors  classified all essentially normal operators  acting on an infinite dimensional separable Hilbert space  with essential spectrum $X$ up to essentially unitarily equivalence by proving that such an operator $A$ can be identified, precisely, by the set of indices  of the Fredholm operators $A-\lambda I$, where $I$ is the identity operator and $\lambda \in \bbc \setminus X$. Later they converted this classification problem to the classification of all essential extension of $C(X)$ by compact operators as any essential normal operator with essential spectrum $X$ gives rise to  an essential extension 
	\[
	0 \rightarrow \clk \rightarrow C^*(\{N,\clk\}) \rightarrow C(X) \rightarrow 0.
	\] 
	Kasparov (\cite{Kas-1979aa})  extended this concept by considering a group   $Ext(A,B)$  of stable unitary equivalence classes of essential   $C^*$-algebra extensions of $A$ by $B \otimes \clk$, where $A$ is nuclear and separable, and $B$ is separable.  
	However, the group  $Ext(A,B)$  remains, in general,  silent regarding any information  about unitary equivalence classes of such extensions, and therefore, two elements in the same class 
	may have non-isomorphic middle $C^*$-algebras.
	To overcome this issue, Pimsner, Popa and Voiculescu~(\cite{PimPopVoi-1979aa}) constructed another group $\mathrm{Ext}_{\mathrm{PPV}}(Y, A)$ for a nuclear $C^*$-algebra $A$ and a
	finite dimensional compact metric space $Y$,  
	consisting of strong unitary equivalence classes of unital homogeneous 
	extensions of $A$ by $C(Y) \otimes \clk$.

	We say that the $C^*$-algebra of a quotient space $G_q/H_q$  is  $q$-invariant  if for different values of $q \in (0,1)$,  $C(G_q/H_q)$ are isomorphic.  
	In (\cite{HonSzy-2002aa}),  Hong and Szymanski showed that the odd dimensional quantum sphere $C(S_q^{2n+1})=C(SU_q(n+1)/SU_q(n))$ can be obtained by applying  quantum double suspension (QDS)  operation to  $C(\bbbt)$ iteratively, and  as a result, one gets $q$-invariance of $S_q^{2n+1}$. 
	Chakraborty and Sundar \cite{ChaSun-2011aa} exploited this fact to construct finite summable regular nontrivial spectral triples on $C(S_q^{2n+1})$.  	Lance (\cite{Lan-1998aa}) proved $q$-invariance of $SO_q(3)$, which,  to the best of our understanding,   has a flaw.  The author  correctly established the following short exact sequence of $C^*$-algebras:
	\[
	\xi: 0 \rightarrow C(\bbbt)\otimes \clk \rightarrow C(SO_q(3)) \rightarrow C(\bbbt) \rightarrow 0.
	\]
	
	This extension can be equivalently  described in terms of its Busby invariant $\beta: C(\bbbt) \rightarrow Q(C(\bbbt) \otimes \clk)$. However, the author argues that $Q(C(\bbbt) \otimes \clk)$ is just $Q(\clk)$ again, so the extension is still specified up to 
	strong equivalence by an index, 
	which is not true. One way to see this is by the Kunneth theorem, it follows that  the group $\mbox{Ext}(C(\bbbt), C(\bbbt))=KK^1(C(\bbbt), C(\bbbt))$ is isomorphic to $\bbz^2$, not $\bbz$, and hence all extensions can't be distinguished by an integer as claimed in \cite{Lan-1998aa}.  Though flawed, the argument clearly suggests the need for a group, which is based on  unitary equivalence classes of essential extensions.  In (\cite{Sau-2019aa}), Saurabh  used such a  group, namely,  $\mathrm{Ext}_{\mathrm{PPV}}(\bbbt,C(S_0^{2\ell+1}))$ group, showed that $K_0$-group is the complete invariant of the middle $C^*$-algebras of  extensions in the group, and thus,  proved $q$-invariance of quantum quaternion spheres.    Recently,  Giselsson (\cite{Gis-2020aa}) proved $q$-invariance of $G_q$ for any connected semi-simple compact Lie group $G$. However, for quotient spaces, the problem  remains open. 
	
	Let us explain the main findings of the paper and its organization. 
	In Sect.$~2$, we begin with an overview of the quantum group $SO_q(\mathcal{N})$ using FRT approach as given in \cite{KliSch-1997aa}.  We discuss irreducible representations of $C(SO_q(\mathcal{N}))$ using the dual pairing between  $U_q(\mathfrak{so}_\mathcal{N})$ and $\mathcal{O}(SO_q(\mathcal{N}))$.   In Sect.$~3$, we recall the $U_q(\mathfrak{so}_\mathcal{N})$-module structure on the  quotient space $SO_q(\mathcal{N})/SO_q(\mathcal{N}-2)$. Applying  the Zhelobenko branching rule in two steps, we  obtain the multiplicity of each co-representation of $SO_q(\mathcal{N})$ occurring in $C(SO_q(\mathcal{N})/SO_q(\mathcal{N}-2))$. Further, we explicitly produce a basis of the subspace spanned by all spectral spaces corresponding to each co-representation. Using that, we establish that the quotient space  $C(SO_q(\mathcal{N})/SO_q(\mathcal{N}-2))$ is generated by the matrix entries of the first and the last row of the generating fundamental matrix of $SO_q(\mathcal{N})$.	 Furthermore, we explicitly list all of its irreducible representations and obtain its faithful representation.   In the next section, we computed the $K$-groups with explicit generators. It is worth mentioning that the $K$-groups of $C(SO_q(\mathcal{N})/SO_q(\mathcal{N}-2))$ are known, thanks to the $KK$-equivalence with its classical counterpart (see \cite{NesTus-2012ab}).  However, here, we obtain their generators explicitly, which could be helpful in many situations, for example, in finding $K$-theory-$K$-homology pairing through index computation.  In Sect. 5, we define $m$-torsioned quantum double suspension of a unital $C^*$-algebra, obtain its representation theory, and $K$-groups.   In Sect. 6,  we show $q$-invariance of some homogeneous spaces of type $B$ and $D$. First,  using  the description of elements of $\mathrm{Ext}_{\mathrm{PPV}}(\bbbt,C(\bbbt))$,  we show that  $C(SO_q(3))$ is isomorphic to $2$-torsioned quantum double suspension of $C(\bbbt)$. This proves  $q$-invariance of $C(SO_q(3))$, and  rectify Lance's argument \cite{Lan-1998aa}.  Using continuous functional calculus, we show that $C(SO_q(4)/SO_q(2))=C(SO_0(4)/SO_0(2))$, where the latter $C^*$-algebra is defined as the $C^*$-algebra generated by the operators obtained as the limit of the standard  generators $C(SO_q(4)/SO_q(2))$ as $q$ approaches to $0$. 
	In both cases, we obtain a $C^*$-algebra at $q=0$, which can be thought of as the crystal limit of such spaces and it would be interesting to see whether it is the same $C^*$-algebra as obtained in \cite{MatYun-2023aa}. 
	Finally, using homogeneous $C^*$-extension theory, and a homotopy argument,  we prove $q$-independence of the $C^*$-algebras $C(SO_q(5)/SO_q(3))$, and $C(SO_q(6)/SO_q(4))$. 
	All these results will likely shed further insight into the theory of crystallization of $C^*$-algebras 
	(see \cite{MatYun-2023aa},\cite{GirPal-2023aa}).

	\textbf{Notations:}
	Let $\mathbb{T}$ denote the set of complex numbers whose modulus is $1$, and let $q$ be a real number lying in the interval $(0,1)$. Let $\mathbb{N}_{0}=\mathbb{N}\cup \{0\}$. The standard bases of the Hilbert spaces $\ell^2(\mathbb{N}_0)$ and
	$\ell^2(\mathbb{Z})$ will be denoted by $\{e_n : n \in \mathbb{N}_0\}$ and $\{e_n : n \in \mathbb{Z}\}$ respectively. The length of a Weyl word $w$ is denoted by $\ell(w)$. The number operator $e_n\mapsto ne_n$ is denoted by $N$. The operator $S$ denotes the left shift  $e_n\mapsto e_{n-1}$. The $C^{\ast}$-subalgebra of $\cll(\ell^2(\mathbb{N}_0))$ generated by $S$ is denoted by $\scrt$. 
	Let $\sigma: \scrt\longrightarrow\mathbb{C}$
	be the homomorphism for which $\sigma(S)=1$. The multiplier algebra and the corona algebra of a $C^*$-algebra $A$ is denoted  by $M(A)$ and   $Q(A)$, respectively.  For a compact Hausdorff space $X$,   the symbol $M(X)$ and $Q(X)$ represent the $C^*$-algebras $M(C(X)\otimes \clk)$ and $Q(C(X)\otimes \clk)$, respectively.   The greatest integer less or equal to $x$ is written as $\left\lfloor x\right\rfloor$.  For $i,j\in \mathbb{N}_0$, the operator $p_{ij}$ denotes the operator on $\ell^2(\mathbb{N}_0)$ which sends the basis element $e_i$ to $e_j$. Throughout the paper, the symbol $\mathcal{N}$ denotes the numbers $2n+1$ and $2n$ for type $B_n$ and $D_n$, respectively.

	\section{Preliminaries} 
	We begin by recalling some key aspects of the compact quantum group $C(SO_q(\mathcal{N}))$ as detailed in \cite[Chapter 9]{KliSch-1997aa}.
	\subsection{The $C^*$-algebra $C(SO_q(\mathcal{N}))$}
	In this subsection, we provide a brief overview of the Hopf algebra structure associated with the compact quantum group $SO_q(\mathcal{N})$ as introduced in \cite[Section 9.3.3]{KliSch-1997aa}. We begin by introducing key notations. Let $\mathcal{N}\geq 3$, and for $1\leq i,j,m,n\leq \mathcal{N}$ we define
	\begin{IEEEeqnarray}{rCl} 
		&i'=& \mathcal{N}+1-i,
		~~\rho_{i}  =  \mathcal{N}/2-i \text{ if } i < i',
		~~\rho_{i'} = -\rho_i \text{ if } i \leq i',
		~~C_j^i =  \delta_{ij'} q^{-\rho_i},\nonumber\\
		&R_{mn}^{ij}=&\begin{cases}
			(q-q^{-1})(\delta_{jm}\delta_{in}- C_{i}^{j}C_{n}^{m}) & \mbox{ if } i>m, \cr
			q^{\delta_{ij}-\delta_{ij^{'}}}\delta_{im}\delta_{jn}  & \mbox{ if } i\leq m, \cr
		\end{cases},\nonumber
	\end{IEEEeqnarray}
	where $\delta_{ij}$ represents the Kronecker delta function. Let $A(R)$ be the unital associative algebra generated by $v_j^i$, $i,j=1,2,\ldots, \mathcal{N}$, subject to the following relations:
	\begin{IEEEeqnarray}{rCl} \label{Relation 2.1}
		\sum_{k,l=1}^{\mathcal{N}}R_{kl}^{ji}v_{s}^{k}v_{t}^{l}-R_{st}^{lk}v_{k}^{i}v_{l}^{j}=0,
		\quad i,j,s,t=1,2,\ldots , \mathcal{N}.
	\end{IEEEeqnarray}
	The matrices $(\!(v_j^i)\!)$ and $(\!(C_j^i)\!)$ are denoted as $V$ and $C$, respectively. Define $J$ as the two-sided ideal of $A(R)$ generated by the entries of the matrices $VCV^t C^{-1}-I$ and $CV^t C^{-1}V-I$. Let $\mathcal{O}\left(O_q(\mathcal{N})\right)$ denotes the quotient algebra $A(R)/J$. The Hopf $*$-algebra structure on $\mathcal{O}(O_{q}(\mathcal{N}))$ comes from the following maps.
	\begin{IEEEeqnarray*}{lll}
		&\bullet \, \text{Comultiplication}
		:\Delta(v_{l}^{k}) = \sum_{i=1}^{\mathcal{N}}v_{i}^{k} \otimes v_{l}^{i}, \quad \quad & \bullet \, \text{Counit}:
		\epsilon(v_{l}^{k}) = \delta_{kl}, \\
		& \bullet\, \text{Antipode}:
		S(v_l^k) = q^{\rho_k-\rho_l}v_{l'}^{k'}, &\bullet \, \text{Involution}:
		(v_{l}^{k})^* = q^{\rho_k - \rho_l}v_{l'}^{k'}.
	\end{IEEEeqnarray*}
	
	Let $\mathcal{D}_q$ be the quantum determinant of the matrix $V$ for the quantum group $O_q(\mathcal{N})$ \cite[Chapter 9, Definition 10]{KliSch-1997aa}. Denote the quotient $\mathcal{O}(O_q(\mathcal{N}))/\langle \mathcal{D}_q - 1\rangle$ of $\mathcal{O}(O_q(\mathcal{N}))$ by the two-sided ideal $\langle\mathcal{D}_q - 1\rangle$ as $\mathcal{O}(SO_q(\mathcal{N}))$. The Hopf $*$-algebra structure on $\mathcal{O}(SO_q(\mathcal{N}))$ is induced from $\mathcal{O}(O_q(\mathcal{N}))$. In $\mathcal{O}\left(SO_q(\mathcal{N})\right)$, the relation $V^*=CV^t C^{-1}$ leads to the following:
	\begin{equation}\label{Relation 2.2}
		VV^* = V^*V = I.
	\end{equation}
	
	The algebra $\mathcal{O}\left(SO_q(\mathcal{N})\right)$ becomes a normed $*$-algebra with the norm defined by
	$$\|a\|=\sup\left\{\|\pi(a)\|: \pi \text{ is a representation of } \mathcal{O}\left(SO_q(\mathcal{N})\right)\right\} \text{ for } a\in \mathcal{O}(SO_q(\mathcal{N})).$$

	Using relation (\ref{Relation 2.2}), we conclude that $\left\|v_j^i\right\|\leq 1$. Therefore,  for all $a\in\mathcal{O}\left(SO_q(\mathcal{N})\right)$, we get $\|a\| < \infty$. We denote the completion of $\mathcal{O}\left(SO_q(\mathcal{N})\right)$ as $C\left(SO_q(\mathcal{N})\right)$. The pair $\left(C\left(SO_q(\mathcal{N})\right),\Delta\right)$ forms a compact quantum group known as a $q$-deformation of the group $SO(\mathcal{N})$. The quantum group $SO_q(2)$ is defined as the circle group $\mathbb{T}$. For details, please see \cite[Section 9.3.3]{KliSch-1997aa}.
	\subsection{Dual pairing between Hopf algebras}
	In this subsection, we recall from \cite{KliSch-1997aa} the dual pairing between $U_{q}(\mathfrak{g})$ and $\mathcal{O}(G_q)$.
	
	\bthm{\rm{\cite{KliSch-1997aa}}}\label{pair}
	For $\mathcal{N}\geq 3$, 
	there exists a unique nondegenerate dual pairing between $U_{q^{1/2}}(\mathfrak{so}_{2n+1})$  and $\mathcal{O}(SO_q(2n+1))$, $U_{q}(\mathfrak{so}_{2n})$  and $\mathcal{O}(SO_q(2n))$,  
	such that
	\begin{align}\label{eqpair}
		\left\langle f,v_j^i\right\rangle=t_{ij}(f), \quad i,j=1,2,\hdots \mathcal{N},
	\end{align}
	where $t_{ij}(f)$ be the matrix entries of $T_1(f)$, and $T_1$ is the vector representation of $U_{q^{1/2}}(\mathfrak{so}_{2n+1})$ and $U_{q}(\mathfrak{so}_{2n})$,  respectively. Moreover, there is a nondeenerate pairing  between $U_{q}(\mathfrak{sl}_2)$ and $\mathcal{O}\left(SL_q(2)\right)$ given by 
	\begin{align}\label{eqpair 1}
		\left\langle f,v_j^i\right\rangle=t_{ij}(f), \quad i,j=1,2,
	\end{align}
	where $t_{ij}(f)$ be the matrix entries of $T_1(f)$, and $T_1$ is the vector representation of $U_{q}(\mathfrak{sl}_2)$. 
	\ethm

	We will explicitly describe $T_1$ for these three cases. Let $E_i,~ F_i, ~K_i$
	and $K_i^{-1}$ be generators of  $U_{q^{1/2}}(\mathfrak{so}_{2n+1})$, $U_{q}(\mathfrak{so}_{2n})$, and $U_{q}(\mathfrak{sl}_2)$, where $i=1,\hdots,n$. For more details, please refer to \cite[Section 6.1.2]{KliSch-1997aa}.
	Define $I_{i,j}$ as a $\mathcal{N} \times \mathcal{N}$ 
	matrix with $1$ in the $(i,j)^{th}$ position and $0$ elsewhere, and $D_{j}$ as a diagonal matrix with $q$ in the $(j,j)^{th}$ 
	position and $1$ elsewhere on the diagonal.
	\begin{itemize}
		\item 
		For the QUEA $U_{q^{1/2}}(\mathfrak{so}_{2n+1})$, we have
		
		\[ \left.
		\begin{array}{rcl}
			T_1(K_i) &=& D_i^{-1}D_{i+1}D_{2n-i+1}^{-1}D_{2n-i+2}, \nonumber\\
			T_1(E_i) &=& I_{i+1,i}-I_{2n-i+2,2n-i+1},\nonumber\\
			T_1(F_i) &=& I_{i,i+1}-I_{2n-i+1,2n-i+2},\nonumber
		\end{array} \right\}\quad  \mbox{for} \quad i \in \{1,2,\hdots,n-1\},
		\]
		and for $i=n$,
		\[ 
		T_1(K_n) = D_n^{-1}D_{n+2},~~
		T_1(E_n) = c(I_{n+1,n}-q^{1/2}I_{n+2,n+1}),~~
		T_1(F_n) = c(I_{n,n+1}-q^{-1/2}I_{n+1,n+2}),
		\]
		where $c=(q^{1/2}+q^{-1/2})^{1/2}$.
		\item 
		For the QUEA $U_{q}(\mathfrak{so}_{2n})$, one has
		\[ \left.
		\begin{array}{rcl}
			T_1(K_i) &=& D_i^{-1}D_{i+1}D_{2n-1}^{-1}D_{2n-i+1}. \nonumber\\
			T_1(E_i) &=& I_{i+1,i}-I_{2n-i+1,2n-i}.\nonumber\\
			T_1(F_i) &=& I_{i,i+1}-I_{2n-i,2n-i+1}.\nonumber
		\end{array} \right\}\quad  \mbox{for} \quad i \in \{1,2,\hdots,n-1\}
		\]
		and for $i=n$, 
		\[ 
		\begin{array}{rcl}
			T_1(K_n) = D_{n-1}^{-1}D_n^{-1}D_{n+1}D_{n+2}, 
			~T_1(E_n) = -I_{n+2,n}+I_{n+1,n-1}, 
			~T_1(F_n) = -I_{n,n+2}+I_{n-1,n+1}. \nonumber
		\end{array} 
		\]
		\item For $U_{q}(\mathfrak{sl}_{2})$, we have
		\begin{IEEEeqnarray}{rCl}
			T_1(K) =   \left({\begin{matrix}
					q^{-1} & 0\\
					0& q\\
			\end{matrix} } \right),\qquad
			T_1(E) =  \left( {\begin{matrix}
					0 & 0\\
					1& 0\\
			\end{matrix} } \right), \qquad
			T_1(F) =  \left( {\begin{matrix}
					0 & 1\\
					0& 0\\
			\end{matrix} } \right).\nonumber
		\end{IEEEeqnarray}
	\end{itemize}	 
	
	We will utilize these pairings to write down irreducible representations of $\mathcal{O}(\text{SO}_q (\mathcal{N}))$, which can be extended to $C(\text{SO}_q (\mathcal{N}))$ to obtain elementary representations of $C(\text{SO}_q (\mathcal{N}))$.
	\subsection{Irreducible Representation of $C(SO_{q}(\mathcal{N}))$}\label{IrRe}
	Let $\Pi$ denote the set  $\{\alpha_1, \alpha_2, \ldots, \alpha_n\}$ consisting of simple roots of $\mathfrak{so}_{\mathcal{N}}$. To keep the notations simple, we denote the root $\alpha_i$ as $i$ and the reflection $s_{\alpha_i}$ defined by the root $\alpha_i$ as $s_i$. The Weyl group $W_n$ of $\mathfrak{so}_{\mathcal{N}}$ can be represented as the group generated by the reflections $\{s_i : 1 \leq i \leq n\}$.\\

	\noindent	\textbf{Elementary representation of} $C(SO_{q}(\mathcal{N}))$: For $1\leq i \leq n$, let $d_i=\left\langle \alpha_i, \alpha_i\right\rangle$ and $q_i=q^{d_i}$. Let $K$, $E$, and $F$ be the standard generators of $U_{q_i}(\mathfrak{sl}_2)$. Let  $\Psi_i : U_{q_i}(\mathfrak{sl}_2) \longrightarrow U_{q}(\mathfrak{so}_\mathcal{N})$ be a homomorphism given on generators by,

	$$\Psi_i(K)=K_i,~~\Psi_i(E)=E_i,~~\Psi_i(F)=F_i .$$
	By duality, there exist a surjective homomorphism
	\begin{displaymath}
		\Psi_i^* :C(SO_{q}(\mathcal{N})) \longrightarrow C(SU_{q_i}(2))
	\end{displaymath}
	given by
	\begin{displaymath}
		\left\langle f, \Psi
		_i^*(v_l^k)\right\rangle = \left\langle \Psi_i(f), v_l^k\right\rangle,
	\end{displaymath}
	where $\left\langle\cdot,\cdot\right\rangle$ given by equation ($\ref{eqpair}$).  Consider the matrix $\left( \begin{smallmatrix}
		u_1^1 & u_2^1\\
		u_1^2 & u_2^2
	\end{smallmatrix}\right)$  whose entries are generators of $\mathcal{O}(SU_{q}(2))$. Let $\pi$ represent the following representation of $C(SU_{q}(2))$ on $\ell^2(\mathbb{N}_0)$:\\
	\[
	\pi(u_l^k)=\begin{cases}
		\sqrt{1-q^{2N+2}}S & \mbox{ if } k=l=1,\cr
		S^{*}\sqrt{1-q^{2N+2}} & \mbox{ if } k=l=2,\cr
		-q^{N+1} & \mbox{ if } k=1,~l=2,\cr
		q^N & \mbox{ if } k=2,~l=1.\cr           \end{cases}
	\] 
	For each $i=1,2,\hdots,n$, define a map $\pi_{s_{i}} = \pi \circ \Psi_{i}^{*}$ of   
	$C(SO_{q}(\mathcal{N}))$. Each $\pi_{s_{i}}$ is an irreducible elementary representation of     
	$C(SO_{q}(\mathcal{N})).$  For $t=(t_{1},t_{2},\hdots ,t_{n}) \in \bbbt^{n}$, define the one dimensional representation  $\tau_t: C(SO_{q}(\mathcal{N})) \longrightarrow \bbc $ by
	\[
	\tau_{t}(v_j^i)=\begin{cases}
		\overline{t_{i}}\delta_{ij} & \mbox{ if } i \leq n,\cr
		t_{\mathcal{N}+1-i}\delta_{ij} & \mbox{ if } i > n.\cr
	\end{cases}
	\]
	For any two representations $\phi_1$ and $\phi_2$ of $C(SO_q (\mathcal{N}))$, define a representation $\phi_1 \ast \phi_2 :=(\phi_1 \otimes \phi_2)\circ\Delta$. For $w\in W_n$ such that $s_{i_1}s_{i_2}\cdots s_{i_n}$ is a reduced form for $w$ and $t\in\mathbb{T}^n$, we define a representation $\pi_{t,w}$ by $\tau_t\ast\pi_{s_{i_1}}\ast\pi_{s_{i_2}}\ast\cdots\ast\pi_{s_{i_n}}$. When $t=1$, we denote the representation $\pi_{t,w}$ by $\pi_{w}$.

	\begin{thm}\cite{KorSoi-1998aa}\label{1.2}
		Let $t\in\mathbb{T}^n$ and $w\in W_n$. Then the representation $\pi_{t,w}$ of $C(SO_q(\mathcal{N}))$ is irreducible. Moreover, two representations, $\pi_{t_1,w_1}$ and $\pi_{t_2,w_2}$,  are equivalent if and only if $t_1=t_2$ and $w_1=w_2$.
	\end{thm}
	Also, we have a homomorphism  $\chi_{w}:C(SO_q(\mathcal{N}))\longrightarrow C(\mathbb{T}^n)\otimes \scrt^{\otimes \ell(w)}$ such that $$\chi(w) (a)(t) = \pi_{t,w }(a), \mbox{ for all } a\in C(SO_q(\mathcal{N})).$$
	
	\bthm\label{1.3}
	Let $v$ be the longest element in the Weyl group $W_n$. Then the homomorphism $\chi_{v}:C(SO_q(\mathcal{N}))\longrightarrow C(\mathbb{T}^n)\otimes \scrt^{\otimes \ell(\vartheta)}$ is  faithful.
	\ethm
	
	We omit  proof of the above Theorem.
	\newsection{The quotient space $C(SO_{q}(\mathcal{N})/SO_{q}(\mathcal{N}-2))$}
	In this section, our aim is to prove that the $C^*$-algebra $C(SO_{q}(\mathcal{N})/SO_{q}(\mathcal{N}-2))$ is the $C^*$-subalgebra of $C(SO_q(\mathcal{N}))$ generated by the elements $\left\{v_m^1, v_m^{\mathcal{N}} :
	m \in \{1,2,\ldots,\mathcal{N}\}\right\}$.

	Let $\vartheta_n$ denote the longest word in the Weyl group $W_n$. We realize $W_{n-1}$, the Weyl group of $\mathfrak{so}_{\mathcal{N}-2}$, as a subgroup of $W_n$ generated by simple reflections $s_2, s_3, \ldots, s_n$. Also, noting that the longest word $\vartheta_{n-1}$ in $W_{n-1}$ is a subword of $\vartheta_n$. 
	We define a mapping $\eta_\mathcal{N}:~ C(SO_{q}(\mathcal{N}))\rightarrow C(SO_{q}(\mathcal{N}-2))$ as follows:
	\begin{align*}	
		&\eta_{\mathcal{N}}(v_{j}^{i})=\begin{cases}
			u_{j-1}^{i-1}, 
			& \mbox{ if } i,j \notin \{1,\mathcal{N}\}, \cr
			\delta_{ij}, & \mbox{ otherwise } \cr
		\end{cases}
	\end{align*}
	where $u_{j}^{i}$ are generators of $C(SO_{q}(\mathcal{N}-2))$, and $\Delta\eta_\mathcal{N} = (\eta_\mathcal{N}\otimes \eta_\mathcal{N})\Delta$. To show that $\eta_\mathcal{N}$ is a $C^{*}$-epimorphism, we'll proceed with odd and even cases.
	
	\begin{itemize}
		\item Case I: Consider $\mathcal{N}=2n+1$.\\
		We view $\chi_{\vartheta_n}(C(SO_{q}(2n+1)))$ as a $C^{\ast}$-subalgebra of $C(\mathbb{T}^{n})\otimes \scrt^{\otimes n^{2}}$. Let $\phi_{2n+1}:=1^{\otimes n-1}\otimes \mathrm{ev}_{1}\otimes 1^{\otimes (n-1)^{2}}\otimes \sigma^{\otimes (2n-1)}$ be the homomorphism from $C(\mathbb{T}^{n})\otimes \scrt^{\otimes n^{2}}$ to $\chi_{\vartheta_n}(C(SO_{q}(2n+1)))$. Restricting $\eta_{2n+1}$ to $\phi$ on $ \chi_{\vartheta_n}(C(SO_{q}(2n+1)))$, we have
		\[
		\eta_{2n+1}(\chi_{\vartheta_{n}}(v_{j}^{i}))=\begin{cases}
			\chi_{\vartheta_{n-1}}(u_{j}^{i}), 
			& \mbox{ if } i \neq 1 \mbox{ or } 2n+1, \mbox{ or } j \neq 1 \mbox{ or } 2n+1, \cr
			\delta_{ij}, & \mbox{ otherwise, } \cr
		\end{cases}
		\] 
		the image of the restriction map equal to $\chi_{\vartheta_{n-1}}(C(SO_{q}(2n-1)))$.
		
		\item Case II: Take $\mathcal{N}=2n$.\\
		Define $\eta_{2n}$ as the restriction of $\phi_{2n}:=1^{\otimes n-1}\otimes \mathrm{ev}_{1}\otimes 1^{\otimes (n^2-3n+3)}\otimes \sigma^{\otimes (2n-2)}$ to $\chi_{\vartheta_{n}}(C(SO_{q}(2n-1)))$, which is contained in $C(\mathbb{T}^{n})\otimes \scrt^{\otimes (n^{2}-n-1)}$. Note that $\vartheta_{n}$ is the longest word in the Weyl group $W_n$ associated with $\mathfrak{so}_{2n}$. 
		
		\[
		\eta_{2n}(\chi_{\vartheta_{n}}(v_{j}^{i}))=\begin{cases}
			\chi_{\vartheta_{n-1}}(u_{j}^{i}), 
			& \mbox{ if } i \neq 1 \mbox{ or } 2n, \mbox{ or } j \neq 1 \mbox{ or } 2n, \cr
			\delta_{ij}, & \mbox{ otherwise, } \cr
		\end{cases}
		\] 
		where $u_{j}^{i}$ are generators of $C(SO_{q}(2n-2))$.
	\end{itemize}

	The quotient space $C(SO_{q}(\mathcal{N})/SO_{q}(\mathcal{N}-2))$ is defined by
	\[
	C(SO_{q}(\mathcal{N})/SO_{q}(\mathcal{N}-2)) = \left\{a\in C(SO_{q}(\mathcal{N})) : (\eta_\mathcal{N}\otimes \mathrm{id})\Delta(a) = I\otimes a\right\}. 
	\]
	We have the co-multiplication action on $C(SO_{q}(\mathcal{N})/SO_{q}(\mathcal{N}-2))$ induced by the compact quantum group $C(SO_q(\mathcal{N}))$:
	\begin{IEEEeqnarray}{rCl}
		C(SO_{q}(\mathcal{N})/SO_{q}(\mathcal{N}-2)) &\longrightarrow & C(SO_{q}(\mathcal{N})/SO_{q}(\mathcal{N}-2)) \otimes C(SO_q(\mathcal{N})) \nonumber \\
		a &\longmapsto & \Delta a. \nonumber 
	\end{IEEEeqnarray}
	
	Using   \cite[Theorem 1.5]{Pod-1995aa} of Podles, we will get that
	\begin{align}\label{deomposition}
		C(SO_{q}(\mathcal{N})/SO_{q}(\mathcal{N}-2))= \overline{\oplus_{\alpha \in \widehat{SO_q(\mathcal{N})}}\oplus_{i \in I_{\alpha} }W_{\alpha,i}}, 
	\end{align}
	where $\alpha$ is a finite-dimensional irreducible co-representation $u^{\alpha}$ of $SO_q(\mathcal{N})$, $ W_{\alpha, i}$ corresponds to 
	$u^{\alpha}$ for all $i \in I_{\alpha}$, and $I_{\alpha}$ denotes the multiplicity of $u^{\alpha}$. According to \cite[Section 7]{KliSch-1997aa}, there is a one-to-one correspondence between  the finite-dimensional irreducible co-representations of $SO_q(\mathcal{N})$ and $n$-tuples of integers $\lambda^{\mathcal{N}}=(\lambda_1,\hdots,\lambda_{\left\lfloor\frac{\mathcal{N}}{2}\right\rfloor})$ satisfying the following inequalities:
	\begin{itemize}
		\item For $\mathcal{N}=2n+1$:
		\[
		\lambda_1 \geq \lambda_2 \geq \cdots \geq \lambda_{\left\lfloor\frac{\mathcal{N}}{2}\right\rfloor}\geq 0.
		\]
		\item For $\mathcal{N}=2n$:
		\[
		\lambda_1 \geq \lambda_2 \geq \cdots \geq |\lambda_{\left\lfloor\frac{\mathcal{N}}{2}\right\rfloor}|.
		\]
	\end{itemize}
	
	The $n$-tuple $\lambda^{\mathcal{N}}$ is referred to as the highest weight, and $V(\lambda^{\mathcal{N}})$ denotes the finite dimensional irreducible representation corresponding to  $\lambda^{\mathcal{N}}$. For notational convenience, we denote the multiplicity 
	$I_{V(\lambda^{\mathcal{N}})}$ by $I_{\lambda^{\mathcal{N}}}$.

	\bppsn \label{Multiplicity Calculation}
	Let $V(\lambda^{\mathcal{N}})$ be a finite-dimensional irreducible representation of $SO_q(\mathcal{N})$  corresponding to highest weight  $\lambda^{\mathcal{N}}=(\lambda_1,\hdots,\lambda_{\left\lfloor\frac{\mathcal{N}}{2}\right\rfloor})$. Then we have 
	\begin{itemize}
		\item for $\mathcal{N}>4:$	\[
		I_{\lambda^{\mathcal{N}}}= \begin{cases}
			\lambda_1-\lambda_2+1, & \quad \text{if} \quad \lambda_i=0 \quad \text{for all}\quad  i \geq 3, \\
			0, & \quad \text{otherwise}.
		\end{cases}
		\]
		\item  for $\mathcal{N}=4:$
		\[	 I_{\lambda^{\mathcal{N}}} =\lambda_1-|\lambda_2|+1.\]
	\end{itemize} 
	\eppsn 
	\begin{proof} The restriction of $V(\lambda^{\mathcal{N}})$ to   $SO_q(\mathcal{N}-2)$ is isomorphic to a direct sum of irreducible finite-dimensional representations $V(\beta^{\mathcal{N}-2})$, where $\beta^{\mathcal{N}-2} =(\beta_1,\cdots , \beta_{\left\lfloor\frac{\mathcal{N}}{2}\right\rfloor-1})$. Each finite dimensional irreducible representation $V(\beta^{\mathcal{N}-2})$  occurring  with certain multiplicities denoted as $m_{\lambda^{\mathcal{N}}}(\beta^{\mathcal{N}-2})$.   By  \cite[Theorem 1.7]{Pod-1995aa}, we have 
		$$ I_{\lambda^{\mathcal{N}}}=m_{\lambda^{\mathcal{N}}}(0).$$
		Here $0$ denote the highest weight of the trivial representation. 
		For the computation of  $m_{\lambda^{\mathcal{N}}}(0)$, we use the classical Zhelobenko  branching rule. 
		
		Let $W(\lambda^{\mathcal{N}})$ denote the irreducible finite dimensional representations of $SO(\mathcal{N})$ with highest weight $\lambda^{\mathcal{N}}$. We apply the classical branching rule in two steps.	 In the first step, we restrict irreducible finite dimensional representation $W(\lambda^{\mathcal{N}})$ to  $SO(\mathcal{N}-1)$ that decompose $W(\lambda^{\mathcal{N}})$ into disjoint sub-representations of $SO(\mathcal{N})$, denote it by $W(\eta)$, where $\eta=(\eta_1,\eta_2\hdots,\eta_{\left\lfloor\frac{\mathcal{N}}{2}\right\rfloor})$. Note that each of the sub-representation occurring only once. In the second step, we again restrict sub-representation $W(\eta)$ to $SO(\mathcal{N}-2)$ that decompose $W(\eta)$ into disjoint sub-representations of $SO(\mathcal{N}-1)$, denote it by $W(\beta)$ with multiplicity one, where $\beta=(\beta_1,\beta_2,\hdots,\beta_{\left\lfloor\frac{\mathcal{N}}{2}\right\rfloor-1})$.  Let us proceed with case by case. 
		\begin{itemize}
			\item For $\mathcal{N}=2n+1$: \quad
			$W(\lambda^{\mathcal{N}})|_{SO(\mathcal{N}-1)}=\oplus_{\lambda_1 \geq \eta_1 \geq \cdots \geq \lambda_{\left\lfloor\frac{\mathcal{N}}{2}\right\rfloor} \geq |\eta_{\left\lfloor\frac{\mathcal{N}}{2}\right\rfloor}|} W(\eta)$.
			We again restrict  $W(\eta)$ to $SO(\mathcal{N}-2)$, using branching rule, we have
			$$W(\eta)|_{SO(\mathcal{N}-2)}=\oplus_{ \eta_1 \geq \beta_1 \geq  \cdots \geq \beta_{\left\lfloor\frac{\mathcal{N}}{2}\right\rfloor-1}\geq |\eta_{\left\lfloor\frac{\mathcal{N}}{2}\right\rfloor}| }W(\beta^{\mathcal{N}-2}).$$ Therefore, we get
			$$W(\lambda^{\mathcal{N}})|_{SO(\mathcal{N}-2)}=\oplus_{\lambda_1 \geq \beta_1 \geq \cdots \geq \beta_{\left\lfloor\frac{\mathcal{N}}{2}\right\rfloor-1} \geq \lambda_{\left\lfloor\frac{\mathcal{N}}{2}\right\rfloor}}m_{\lambda_{\mathcal{N}}}(\beta^{\mathcal{N}-2}) W(\beta^{\mathcal{N}-2}).$$
			
			The multiplicity $m_{\lambda^{\mathcal{N}}}(\beta^{\mathcal{N}-1})$ is determined by the number of $n$-tuples of integers  $(\gamma_1, \cdots ,\gamma_{\left\lfloor\frac{\mathcal{N}}{2}\right\rfloor})$ that satisfy the following inequalities:
			\begin{IEEEeqnarray}{rCl}
				\lambda_1 \geq \gamma_1 \geq \lambda_2 \geq \gamma_2 \geq \cdots \geq \lambda_{\left\lfloor\frac{\mathcal{N}}{2}\right\rfloor} \geq \gamma_{\left\lfloor\frac{\mathcal{N}}{2}\right\rfloor}\geq 0, \label{3.1} \\
				\gamma_1 \geq \beta_1 \geq \gamma_2 \geq \beta_2 \cdots  \geq \beta_{\left\lfloor\frac{\mathcal{N}}{2}\right\rfloor-1} \geq \gamma_{\left\lfloor\frac{\mathcal{N}}{2}\right\rfloor}\geq 0.\label{3.2} 
			\end{IEEEeqnarray}
			\item For $\mathcal{N}=2n$:\quad
			$W(\lambda^{\mathcal{N}})|_{SO(\mathcal{N}-1)}=\oplus_{\lambda_1 \geq \eta_1 \geq \cdots \geq \eta_{\left\lfloor\frac{\mathcal{N}}{2}\right\rfloor-1}\geq |\lambda_{\left\lfloor\frac{\mathcal{N}}{2}\right\rfloor}|} W(\eta)$.
			We again restrict  $W(\eta)$ to $SO(\mathcal{N}-2)$, using branching rule, we have
			$$W(\eta)|_{SO(\mathcal{N}-2)}=\oplus_{ \eta_1 \geq \beta_1 \geq  \cdots\geq \eta_{\left\lfloor\frac{\mathcal{N}}{2}\right\rfloor-1}\geq |\beta_{\left\lfloor\frac{\mathcal{N}}{2}\right\rfloor-1}|} W(\beta^{\mathcal{N}-2}).$$ Therefore, we get
			$$W(\lambda^{\mathcal{N}})|_{SO(\mathcal{N}-2)}=\oplus_{\lambda_1 \geq \beta_1 \geq \cdots \geq \beta_{\left\lfloor\frac{\mathcal{N}}{2}\right\rfloor-1} \geq \lambda_{\left\lfloor\frac{\mathcal{N}}{2}\right\rfloor}}m_{\lambda_{\mathcal{N}}}(\beta^{\mathcal{N}-2}) W(\beta^{\mathcal{N}-2}).$$
			
			The multiplicity $m_{\lambda^{\mathcal{N}}}(\beta^{\mathcal{N}-1})$ is determined by the number of $n$-tuples of integers $(\gamma_1, \cdots \gamma_{\left\lfloor\frac{\mathcal{N}}{2}\right\rfloor})$ that satisfy the following inequalities:
			\begin{align}\label{3.3}
				\lambda_1 \geq \gamma_1 \geq \lambda_2 \geq \gamma_2 \geq \cdots \geq |\lambda_{\left\lfloor\frac{\mathcal{N}}{2}\right\rfloor} |\geq \gamma_{\left\lfloor\frac{\mathcal{N}}{2}\right\rfloor},  \\  
				\gamma_1 \geq \beta_1 \geq \gamma_2 \geq \beta_2\geq \cdots  \geq |\beta_{\left\lfloor\frac{\mathcal{N}}{2}\right\rfloor-1}| \geq \gamma_{\left\lfloor\frac{\mathcal{N}}{2}\right\rfloor}. \label{3.4}	
			\end{align}
		\end{itemize}
		
		\begin{itemize}
			\item Case $\mathcal{N}>4$: We take $\beta_i=0$ for all $1\leq i \leq n-1$.  Using equations (\ref{3.1}),(\ref{3.2}), (\ref{3.3}), and (\ref{3.4}), we have
			
			\[
			I_{\lambda^{\mathcal{N}}} =m_{\lambda^{\mathcal{N}}}(0)= \begin{cases}
				\lambda_1-\lambda_2+1, & \quad \text{if} \quad \lambda_i=0 \quad \text{for all}\quad  i \geq 3, \\
				0, & \quad \text{otherwise}.
			\end{cases}
			\]

			\item Case $\mathcal{N}=4$: Using equations ($\ref{3.3})$ and ($\ref{3.4}$), we obtain $	I_{\lambda^4}=m_{\lambda^{4}}(0)=\lambda_1-|\lambda_2|+1$.
		\end{itemize}	
	\end{proof}

	We define $\mathfrak{A}^\mathcal{N}$ to be the $*$-algebra generated by $\left\{v_m^1,v_m^{\mathcal{N}} : m \in \{1,2,\hdots \mathcal{N}\}\right\}$. Using dual pairing, we  define the following action on $\mathfrak{A}^\mathcal{N}$.
	\begin{itemize}
		\item
		For $\mathcal{N}=2n+1$,
		\textbf{$U_{q^{1/2}}(\mathfrak{so}_{2n+1})$-module structure}: The pairing $\left\langle\cdot,\cdot\right\rangle$ given by \eqref{eqpair} induces a $U_{q^{1/2}}(\mathfrak{so}_{2n+1})$-module structure on  $\mathfrak{A}^{2n+1}$ which is as follows:
		\begin{IEEEeqnarray}{rCl} \label{module}
			fv=(1\otimes \left\langle\, f\,,\cdot\,\right\rangle )\Delta(v)= \left\langle f,  v_{(2)}\right\rangle v_{(1)} ; \, \mbox { for } v \in \mathfrak{A}^{2n+1}, \,f \in U_{q^{1/2}}(\mathfrak{so}_{2n+1}),
		\end{IEEEeqnarray}
		where $\Delta(v)=\sum  v_{(1)} \otimes v_{(2)}$ in Sweedler notation. 
		\item
		For $\mathcal{N}=2n$, \textbf{$U_{q}(\mathfrak{so}_{2n})$-module structure}: Replace $\left\langle\, \cdot\,,\cdot\,\right\rangle $ given in (\ref{module}) by second pair in \eqref{eqpair} will give action on  $\mathfrak{A}^{2n}$ .\\
	\end{itemize}
	Let $a = v_{\mathcal{N}-1}^1$, $b = v_{\mathcal{N}-1}^{\mathcal{N}}$, $c = v_{\mathcal{N}}^1$, and $d = v_{\mathcal{N}}^{\mathcal{N}}$. For $n>2$, utilizing the defined module structure, we get the following:
	
	\begin{enumerate}[(1)]
		\item{{\label{Observation 1}}} For $x\in\{a,b\}$ and $y\in\{c,d\}$,
		\[
		K_i(x)= \begin{cases}
			q^{-1}x &\text{ if } i=1,\\
			qx &\text{ if } i=2,\\
			x &\text{ if } i\geq 3,
		\end{cases} \quad\text{and}\quad K_i(y)=\begin{cases}
			qy &\text{ if } i=1,\\
			y &\text{ if } i\geq 2.
		\end{cases}
		\]
		
		\item{{\label{Observation 2}}} For $y\in\{c,d\}$, $E_i(y)=0$ for each $i\in\mathbb{\mathcal{N}}$, and 
		\[
		E_i(a)= \begin{cases}
			-c &\text{ if } i=1,\\
			0 &\text{ if } i\geq 2,
		\end{cases} \quad\text{and}\quad E_i(b)=\begin{cases}
			-d &\text{ if } i=1,\\
			0 &\text{ if } i\geq 2.
		\end{cases}
		\]
		\item{{\label{Observation 3}}} Moreover, using the relation $\Delta^n E_1 = E_1\otimes K_1^{\otimes n} + 1\otimes E_1\otimes K_1^{\otimes (n-1)} + \cdots + 1^{\otimes (n-1)}\otimes E_1\otimes K_1 + 1^{\otimes n}\otimes E_1$, we get 
		\begin{align*}
			E_1(b^nc^n) &=A_1^0b^{n-1}c^{n}d,\\
			E_1(ab^{n-1}c^{n-1}d) &=A_1^1b^{n-1}c^nd+A_2^1ab^{n-2}c^{n-1}d^2,\\
			\cdots &  \cdots\cdots\\
			E_1(a^{n-1}bcd^{n-1})&=A_1^{n-1}a^{n-2}bc^2d^{n-1}+A_2^{n-1}a^{n-1}cd^{n},\\
			E_1(a^nd^n)&=A_1^na^{n-1}cd^n,
		\end{align*}
		where $A_1^{0}=-q^{-2n+1}[n]_q$, $A_1^n=-q[n]_q$, $A_1^i=-q^{n-i+1}[i]_q$, $A_2^i=-q^{3i-2n+1}[n-i]_q$ for $i\in\{1,2,\cdots,n-1\}$ and $[a]_q:=\frac{q^a-q^{-a}}{q-q^{-1}}$.\\
	\end{enumerate}

	\bdfn	 We define a vector $u$ in $\mathfrak{A}^\mathcal{N}$ as the highest weight vector with highest weight $(\lambda_1,\lambda_2,0,\cdots,0)$ if it satisfies the following conditions:
	\begin{IEEEeqnarray}{rCll}
		K_i(u)&=&q^{r_i} u,  &\nonumber \\
		E_i(u)&=&0,  &\mbox{for all} \quad i \in \left\{1,\hdots , n\right\}, \nonumber 
	\end{IEEEeqnarray}
	where $r_i$ is determined as follows:
	\begin{table}[h]
		\centering
		\begin{tabular}{|c|c|c|}
			\hline
			\rowcolor{gray!30}
			\multicolumn{1}{|c|}{\textbf{Case}} & \multicolumn{1}{|c|}{\textbf{$\mathcal{N}=2n$}} & \multicolumn{1}{|c|}{\textbf{$\mathcal{N}=2n+1$}} \\
			\hline
			$n=1$ &$r_1=\lambda_1$  & $r_1=2\lambda_1$ \\
			\hline
			$n=2$ & $r_1=\lambda_1-\lambda_2$, $r_2=\lambda_1+\lambda_2$ & $r_1=\lambda_1-\lambda_2$, $r_2=2\lambda_2$ \\
			\hline
			$n=3$ & $r_1=\lambda_1-\lambda_2$, $r_2=\lambda_2$, $r_3=\lambda_2$ & $r_1=\lambda_1-\lambda_2$, $r_2=\lambda_2$, $r_3=0$ \\
			\hline
			$n>3$ & $r_1=\lambda_1-\lambda_2$, $r_2=\lambda_2$, $r_i=0$ for $i>2$ & $r_1=\lambda_1-\lambda_2$, $r_2=\lambda_2$, $r_i=0$ for $i>2$ \\
			\hline
		\end{tabular}
	\end{table}
	\edfn

	\bppsn \label{pp1}
	For $\mathcal{N}>4$, there exist $\lambda_1-\lambda_2+1$ linearly independent highest weight vectors in $\mathfrak{A}^\mathcal{N}$ with highest weight $(\lambda_1,\lambda_2,0,\hdots,0)$ with $\lambda_1\geq\lambda_2\geq 0$. 
	\eppsn
	
	\begin{proof}
		Define $r = \lambda_1 - \lambda_2$. 	
		Consider the element $\omega_\mathcal{N}$ of the Weyl group of $\mathfrak{so}_\mathcal{N}$ defined by 
		\[
		\omega_\mathcal{N}= \begin{cases}
			s_1 s_2 \cdots s_{\left\lfloor\frac{\mathcal{N}}{2}\right\rfloor -1} s_{\left\lfloor\frac{\mathcal{N}}{2}\right\rfloor} s_{\left\lfloor\frac{\mathcal{N}}{2}\right\rfloor -1}\cdots s_2 s_1, &\text{if } \mathcal{N} \text{ is odd,}\\
			s_1 s_2 \cdots s_{\left\lfloor\frac{\mathcal{N}}{2}\right\rfloor -1} s_{\left\lfloor\frac{\mathcal{N}}{2}\right\rfloor} s_{\left\lfloor\frac{\mathcal{N}}{2}\right\rfloor -2}\cdots s_2 s_1, &\text{if } \mathcal{N} \text{ is even,}
		\end{cases}
		\] 
		and consider the subword $\omega_\mathcal{N}'$ of $\omega_\mathcal{N}$ defined by 
		\[
		\omega_\mathcal{N}'= \begin{cases}
			s_1 s_2 \cdots s_{\left\lfloor\frac{\mathcal{N}}{2}\right\rfloor -1} s_{\left\lfloor\frac{\mathcal{N}}{2}\right\rfloor} s_{\left\lfloor\frac{\mathcal{N}}{2}\right\rfloor -1}\cdots s_2, &\text{if } \mathcal{N} \text{ is odd,}\\
			s_1 s_2 \cdots s_{\left\lfloor\frac{\mathcal{N}}{2}\right\rfloor -1} s_{\left\lfloor\frac{\mathcal{N}}{2}\right\rfloor} s_{\left\lfloor\frac{\mathcal{N}}{2}\right\rfloor -2}\cdots s_2, &\text{if } \mathcal{N} \text{ is even.}
		\end{cases}
		\]
		It is then straightforward to observe that $\pi_{\omega'_\mathcal{N}}(c)=0$. According to Observation ($\ref{Observation 3})$ , we can choose nonzero constants $A_k$'s such that $E_1(b^{\lambda_2}c^{\lambda_2}+A_1ab^{\lambda_2-1}c^{\lambda_2-1}d+\cdots +A_{\lambda_2}a^{\lambda_2}d^{\lambda_2})=0$. Let $u_{\lambda_2}=b^{\lambda_2}c^{\lambda_2}+A_1ab^{\lambda_2-1}c^{\lambda_2-1}d+\cdots +A_{\lambda_2}a^{\lambda_2}d^{\lambda_2}$. Using representation of $SO_q(\mathcal{N})$, we have
		\[
		\pi_{\omega^{'}_\mathcal{N}}(u_{\lambda_2})(e_0\otimes e_0\otimes \cdots \otimes e_0)=
		\pi_{\omega^{'}_\mathcal{N}}(A_{\lambda_2}a^{\lambda_2}d^{\lambda_2})(e_0\otimes e_0\otimes \cdots \otimes e_0)\neq 0
		\]
		which implies that $u_{\lambda_2}\neq 0$. Now, define \begin{align*}	
			&x_i=\begin{cases}
				c^id^{r-i}u_{2\lambda_2}, 
				& \mbox{ if } \mathcal{N}=5, \cr
				c^id^{r-i}u_{\lambda_2}, & \mbox{ if } \mathcal{N}>5, \cr
			\end{cases}
		\end{align*}
		for $i\in\{0,\cdots,r\}$. Utilizing the actions of $K_i$ and $E_i$ computed in Observations (\ref{Observation 1}) and (\ref{Observation 2}), we get that $x_i$'s are elements of $\mathfrak{A}^\mathcal{N}$ with the highest weight $(\lambda_1,\lambda_2,0,\hdots,0)$.  Now, if we look at $n^{\text{th}}$ position of each term of the  $\pi_{\omega_\mathcal{N}}(x_i)(e_0\otimes e_0\otimes \cdots \otimes e_0)$ then we can see  that one term has $n^{\text{th}}$ position $e_{2(r-i)}$ for $\mathcal{N}=2n+1$ and $e_{r-i}$ for $\mathcal{N}=2n$, while the other term has the $n^{\text{th}}$ position say $e_k$ where $k<2(r-i)$ for $\mathcal{N}=2n+1$ and $k<(r-i)$ for $\mathcal{N}=2n$. Hence $x_i$'s are linearly independent.
	\end{proof}
	
	\bppsn \label{pp2}
	There exist $\lambda_1-|\lambda_2|+1$ linearly independent highest weight vectors in $\mathfrak{A}^4$ with highest weight $(\lambda_1,\lambda_2)$ with $\lambda_1\geq|\lambda_2|$.
	\eppsn
	
	\begin{proof}
		We divide the proof in two cases. Firstly, consider the case where $\lambda_2 \geq 0$. It follows from the reasoning presented in Proposition $\ref{pp1}$.
		
		Secondly, we tackle the case where $\lambda_2 < 0$. Here, we define $a = v_{2}^1$, $b = v_{2}^{4}$, $c = v_{4}^1$, $d = v_{4}^{4}$, and $r = \lambda_1 + \lambda_2$. Notably, the action of $E_1$ and $E_2$ on $\mathfrak{A}^4$ interchanges roles  in Observation (\ref{Observation 2}), while the action of $K_1$ and $K_2$ is as follows:
		\[K_1(a) = qa,~ K_1(b) = qb,~ K_1(c) = qc, ~K_1(d) = qd.\]
		\[K_2(a) = q^{-1}a, ~K_2(b) = q^{-1}b, ~K_2(c) = c,~ K_2(d) = d. 
		\]
		Define $x_i = c^id^{r-i}u_{-\lambda_2}$ for $i \in \{0,\hdots,r\}$, where $u_{-\lambda_2}$ is defined similarly to Proposition $\ref{pp1}$ but with the replacement of $\lambda_2$ by its negative sign. Utilizing a similar argument as presented in Proposition $\ref{pp1}$, we get that $u_{-\lambda_2}\neq 0$. By using the actions of $K_1$, $K_2$, and replacing $E_1$ by $E_2$  in Observation (\ref{Observation 3}), it follows that $x_i$'s are highest weight vectors. Let $w_4=s_1s_2$ and using the representation of $SO_q(4)$, we find that $\pi_{w_4}(x_i)(e_0\otimes e_0) = Ce_{r-i-\lambda_2}\otimes e_{r-i}$, where $C$ is a non-zero constant. Hence, we conclude that the $x_i$'s are linearly independent.
	\end{proof}

	\bthm 
	The quotient space $C(SO_{q}(\mathcal{N})/SO_{q}(\mathcal{N}-2))$ is the $C^*$-algebra generated by $\left\{v_i^1,v_i^{\mathcal{N}} : i \in \{1,2,\hdots \mathcal{N}\}\right\}$.
	\ethm
	
	\begin{proof}
		That both $v_{i}^{\mathcal{N}}$ and $v_{i}^{1}=q^{\rho_{1}-\rho_{i}}(v_{\mathcal{N}-i+1}^{\mathcal{N}})^{*}$ belong to $C(SO_{q}(\mathcal{N})/SO_{q}(\mathcal{N}-2))$ for $i=1,2,\hdots,\mathcal{N}$, can be verified by a straightforward calculation. This shows that 
		\[
		C(SO_{q}(\mathcal{N})/SO_{q}(\mathcal{N}-2)) \supseteq C^*\left\{v_i^1,v_i^{\mathcal{N}} : i \in \{1,2,\hdots,\mathcal{N}\}\right\}.
		\]
		To prove the reverse inclusion, define
		\[
		\cla= \oplus_{\alpha \in \widehat{SO(\mathcal{N})}}\oplus_{i \in I_{\alpha} }W_{\alpha,i}.
		\]
		From equation (\ref{deomposition}), it is enough to  show that $\cla \subseteq C^*\left\{v_i^1,v_i^{\mathcal{N}} : i \in \{1,2,\hdots,\mathcal{N}\}\right\}$.
		From Propositions \ref{Multiplicity Calculation}, \ref{pp1}, \ref{pp2}), it follows that 
		$$ \oplus_{i \in I_{\alpha} }W_{\alpha,i} \subseteq \cla, \,\, \mbox{ for all } \alpha \in \widehat{SO(\mathcal{N})}.$$  
		This completes the proof.
	\end{proof}

	\subsection{Irreducible Representation of $C(SO_{q}(\mathcal{N})/SO_q(\mathcal{N}-2))$}
	
	Fix $q^{a_{\mathcal N}} = q$ if $\mathcal N$ is even, and $q^{a_{\mathcal N}} = q^{1/2}$ if
	$\mathcal N$ is odd.
	The matrix coefficients $(\!(x_j^i)\!)$ of the irreducible representation of
	$U_{q^{a_{\mathcal N}}}(\mathfrak{so}_{2n+1})$ with highest weight
	$\left(\tfrac12, \tfrac12, \ldots, \tfrac12\right)$ generate the Hopf-$*$-algebra
	$\mathcal O(\mathrm{Spin}_q(\mathcal N))$.
	The corepresentation of $\mathcal O(\mathrm{Spin}_q(\mathcal N))$ corresponding to
	the highest weight $(1,0,\ldots,0)$ coincides with $(\!(v_j^i)\!)$, whose matrix
	entries generate a proper Hopf-$*$-subalgebra $\mathcal O(SO_q(\mathcal N))$.
	As a consequence,  $C(SO_q(\mathcal N))$ is a $C^*$-subalgebra of
	$C(\mathrm{Spin}_q(\mathcal N))$.  Moreover, both are $C^*$-subalgebra of $C(\bbbt^n)\otimes \scrt ^{\ell(\vartheta_n)}$. 
	Therefore, one can restrict  $\phi_{\mathcal N}$  to $C(\mathrm{Spin}_q(\mathcal N))$,
	which induces a quantum
	group homomorphism
	\[
	\gamma_{\mathcal N} \colon C(\mathrm{Spin}_q(\mathcal N))
	\longrightarrow C(\mathrm{Spin}_q(\mathcal N-2)),
	\]
	establishing $\mathrm{Spin}_q(\mathcal N-2)$ as a quantum subgroup of
	$\mathrm{Spin}_q(\mathcal N)$.
	Note that the restriction of $\gamma_{\mathcal N}$ to $C(SO_q(\mathcal N))$
	coincides with $\eta_{\mathcal N}$.
	The quotient space $\mathrm{Spin}_q(\mathcal N)/\mathrm{Spin}_q(\mathcal N-2)$ is
	defined by
	\[
	C(\mathrm{Spin}_q(\mathcal N)/\mathrm{Spin}_q(\mathcal N-2))
	= \left\{ b \in C(\mathrm{Spin}_q(\mathcal N)) :
	(\gamma_{\mathcal N}\otimes \mathrm{id})\Delta(b) = I \otimes b \right\}.
	\]
	Clearly,
	$C(SO_q(\mathcal N)/SO_q(\mathcal N-2))$ is a $C^*$-subalgebra of
	$C(\mathrm{Spin}_q(\mathcal N)/\mathrm{Spin}_q(\mathcal N-2))$.
	The following result shows that these two $C^*$-algebras coincide.
	
	\bthm \label{S} Let $q \in (0,1)$. Then one has 
	$$C(\mathrm{Spin}_{q}(\mathcal{N})/\mathrm{Spin}_{q}(\mathcal{N}-2))=C(SO_{q}(\mathcal{N})/SO_{q}(\mathcal{N}-2)).$$
	\ethm 
	\prf  	By  \cite[Theorem 1.5]{Pod-1995aa}, we have 
	\begin{align}\label{deomposition1}
		C(\mathrm{Spin}_q(\mathcal N)/\mathrm{Spin}_q(\mathcal N-2))= \overline{\oplus_{\alpha \in \widehat{\mathrm{Spin}_q(\mathcal N)}}\oplus_{i \in I_{\alpha} }W_{\alpha,i}}.
	\end{align}
	Using two step branching rule (\cite{Zhe-1962aa}) as done in Proposition \ref{Multiplicity Calculation}, one can see the multiplicity of a spinor representation with half integer highest weight is zero. Therefore we have 
	$$C(\mathrm{Spin}_q(\mathcal N)/\mathrm{Spin}_q(\mathcal N-2))= \overline{\oplus_{\alpha \in \widehat{SO_q(\mathcal N)}}\oplus_{i \in I_{\alpha} }W_{\alpha,i}}= C(SO_{q}(\mathcal{N})/SO_{q}(\mathcal{N}-2)).$$
	\qed

	\noindent \textbf{Irreducible representation of $C(SO_{q}(2n+1)/SO_q(2n-1))$: } Let $\omega_{k}$ represent the following element of the Weyl group of $\mathfrak{so}_{2n+1}$:
	
	\[
	\omega_k =\begin{cases}
		I & \text{if } k =1, \\
		s_1s_2\cdots s_{k-1} & \text{if } 2 \leq k \leq n+1,\\
		s_1s_2\cdots s_{n-1}s_{n}s_{n-1}\cdots s_{2n-k+1} & \text{if } ~n+1 < k <2n+1,\\
	\end{cases}
	\] 
	
	Define $\phi_{t,w_k}$ to be the restriction of $\pi_{t,w_k}$ to the quotient space $C(SO_{q}(2n+1)/SO_q(2n-1))$, if $k \geq 2$.  For $k=1$, define $\phi_{t,I}:C(SO_{q}(2n+1)/SO_q(2n-1))\rightarrow \bbc$ such that $\phi_{t,I}(v_j^{2n+1}) = t\delta_{1(2n-j+2)}$ for $j\in\{1,2,\hdots,2n+1\}$. 
	\bthm \label{repn}
	The collection $\left\{\phi_{t,\omega_k}: t \in \bbbt, 1\leq k < 2n+1 \right\}$ gives a complete list of irreducible representations of $C(SO_{q}(2n+1)/SO_q(2n-1))$.
	\ethm
	
	\begin{proof}
		It follows from Theorem \ref{S} and Theorem $2.2$ in \cite{NesTus-2012ab}.
	\end{proof}
	
	To obtain a faithful representation of $C(SO_{q}(2n+1)/SO_q(2n-1))$, define 
	\begin{align*}
		\phi_{\omega_k}: &~C(SO_{q}(2n+1)/SO_q(2n-1))\rightarrow C(\bbbt) \otimes \scrt^{\otimes \ell({w_k})}\\
		&\phi_{\omega_k}(a)(t) = \phi_{t,\omega_k}(a) \quad \forall a \in  C(SO_{q}(2n+1)/SO_q(2n-1)).
	\end{align*}
	
	\bcrlre
	$\phi_{\omega_{2n}}$ is a faithful representation of $C(SO_{q}(2n+1)/SO_q(2n-1))$.
	\ecrlre
	
	\begin{proof}
		It is easy to see that any irreducible representation factors through $\phi_{\omega_{2n}}$ as, $\omega_k$ is a subword of $\omega_{2n}$. This proves the claim.
	\end{proof}
	
	We will illustrate these representations through diagrams, with further detailed provided in $\cite{BhuSau-2023aa}$. In these diagrams, each path from node $i$ on the left to node $j$ on the right represents an endomorphism acting on the Hilbert space  given at top of the diagram. Here, the arrows denote operators explicitly defined in  \cite[Section 3]{BhuSau-2023aa}. As an example, consider the case where $n=3$ and $w=s_1s_2s_3s_2$. The representation $\phi_{\omega}$ corresponds to Figure 1.

	\begin{figure}[h]
		\centering
		\resizebox{\textwidth}{!}{
			\hspace{-3.6cm}
			\begin{subfigure}{0.5\textwidth}
				\centering
				\def\labelstyle{\scriptstyle} 
				\xymatrix@C=15pt@R=20pt{
					& & & &  &{}\ar@{}[r]_{~~~\scalemath{0.6}{\mathbf{\ell^2(\mathbb{Z})~\otimes}}}
					& {}\ar@{}[r]_{~\scalemath{0.6}{\mathbf{\ell^{2}(\mathbb{N}_0)~\otimes}}}
					& {}\ar@{}[r]_{~\scalemath{0.6}{\mathbf{\ell^{2}(\mathbb{N}_0)~\otimes}}}
					& {}\ar@{}[rr]_{~~~~~\scalemath{0.6}{\mathbf{\ell^{2}(\mathbb{N}_0)~~~~\otimes}}}
					&& {}\ar@{}[r]_{\scalemath{0.6}{\mathbf{\ell^{2}(\mathbb{N}_0)}}}
					&& \\
					& & & & 
					&7~\circ\ar@{-}[r]^{t}&\circ\ar@{-}[r]^{+}\ar@{-}[rd]^{\;\;-}
					&\circ\ar@{-}[r]&\circ\ar@{-}[rr]&&\circ\ar@{-}[r]
					&\circ ~7\\                                                   
					& & & & 
					&6~\circ\ar@{-}[r]&\circ\ar@{-}[r]_{-}\ar@{-}[ru]_{\;\;-}
					&\circ\ar@{-}[r]^{+}\ar@{-}[rd]^{\;\;-}
					&\circ\ar@{-}[rr]&&\circ\ar@{-}[r]^{+}\ar@{-}[rd]^{\;\;-}
					&\circ ~6& \\
					& & & & 
					&5~\circ\ar@{-}[r]&\circ\ar@{-}[r]&\circ\ar@{-}[r]_{-}\ar@{-}[ru]_{\;\;-}
					&\circ\ar@{-}[rr]^{++}\ar@{->>}[drr]\ar@{-}[ddrr]
					&&\circ\ar@{-}[r]_{-}\ar@{-}[ru]_{\;\;-}
					&\circ ~5& \\
					& & & & 
					&4~\circ\ar@{-}[r]&\circ\ar@{-}[r]&\circ\ar@{-}[r]
					&\circ\ar@{--}[rr]\ar@{=}[urr]\ar@{..>}[drr]&
					&\circ\ar@{-}[r] &\circ ~4&\\
					& & & & 
					&3~\circ\ar@{-}[r]&\circ\ar@{-}[r]
					&\circ\ar@{-}[r]\ar@{-}[r]^{+}\ar@{-}[rd]^{\;\;+}&\circ \ar@{-}[rr]_{--}\ar@{-->}[urr]\ar@{-}[uurr]^{-}  &&\circ\ar@{-}[r]^{+}\ar@{-}[rd]^{\;\;+}
					&\circ ~3&\\
					& & & & 
					&2~\circ\ar@{-}[r]&\circ\ar@{-}[r]^{+}\ar@{-}[rd]&\circ\ar@{-}[r]_{-}\ar@{-}[ru]
					&\circ\ar@{-}[rr]&&\circ\ar@{-}[r]_{-}\ar@{-}[ru]
					&\circ ~2&\\
					& & & & 
					&1~\circ\ar@{-}[r]&\circ\ar@{-}[r]_{-}\ar@{-}[ru]
					&\circ\ar@{-}[r]&\circ\ar@{-}[rr]&&\circ\ar@{-}[r]
					&\circ ~1&\\   
				}
				\caption*{\hspace{5cm}\mbox{Figure 1: Diagram of $\phi_{w}$}}
				\label{fig:phi_w}
			\end{subfigure}
			\hspace{-3.5cm}
			\begin{subfigure}{0.5\textwidth}
				\centering
				\def\labelstyle{\scriptstyle} 
				\xymatrix@C=15pt@R=20pt{
					& & & &  &{}\ar@{}[r]_{~~~~\scalemath{0.6}{\mathbf{\ell^2(\mathbb{Z})\otimes}}}
					& {}\ar@{}[r]_{~~\scalemath{0.6}{\mathbf{\ell^{2}(\mathbb{N}_0)\otimes}}}
					& {}\ar@{}[r]_{~~\scalemath{0.6}{\mathbf{\ell^{2}(\mathbb{N}_0)\otimes}}}
					& {}\ar@{}[r]_{~~\scalemath{0.6}{\mathbf{\ell^{2}(\mathbb{N}_0)\otimes}}}
					& {}\ar@{}[r]_{~~\scalemath{0.6}{\mathbf{\ell^{2}(\mathbb{N}_0)\otimes}}}
					& {}\ar@{}[r]_{\scalemath{0.6}{\mathbf{\ell^{2}(\mathbb{N}_0)}}}
					&\\
					& & & & 
					&8~\circ\ar@{-}[r]^{t}&\circ\ar@{-}[r]^{+}\ar@{->}[rd]
					&\circ\ar@{-}[r]&\circ\ar@{-}[r]&\circ\ar@{-}[r]&\circ\ar@{-}[r]
					&\circ ~8\\                                                   
					& & & & 
					&7~\circ\ar@{-}[r]&\circ\ar@{-}[r]_{-}\ar@{->}[ru]
					&\circ\ar@{-}[r]^{+}\ar@{->}[rd]
					&\circ\ar@{-}[r]&\circ\ar@{-}[r]&\circ\ar@{-}[r]^{+}\ar@{->}[rd]
					&\circ ~7& \\
					& & & & 
					&6~\circ\ar@{-}[r]&\circ\ar@{-}[r]&\circ\ar@{-}[r]_{-}\ar@{->}[ru]
					&\circ\ar@{-}[r]^{+}\ar@{->}[dr]&\circ\ar@{-}[r]\ar@{->}[ddr]
					&\circ\ar@{-}[r]_{-}\ar@{->}[ru]
					&\circ ~6& \\& & & & 
					&5~\circ\ar@{-}[r]&\circ\ar@{-}[r]&\circ\ar@{-}[r]
					&\circ\ar@{-}[r]_{-}\ar@{->}[ur]&\circ\ar@{-}[r]\ar@{-}[ddr]
					&\circ\ar@{-}[r] &\circ ~5&\\
					& & & & 
					&4~\circ\ar@{-}[r]&\circ\ar@{-}[r]&\circ\ar@{-}[r]
					&\circ\ar@{-}[r]^{+}\ar@{-}[dr]&
					\circ\ar@{-}[r] \ar@{->}[uur]	&\circ\ar@{-}[r] &\circ ~4&\\
					& & & & 
					&3~\circ\ar@{-}[r]&\circ\ar@{-}[r]
					&\circ\ar@{-}[r]\ar@{-}[r]^{+}\ar@{-}[rd]&\circ \ar@{-}[r]_{-}\ar@{-}[ur]  &\circ\ar@{-}[r]\ar@{-}[uur]&\circ\ar@{-}[r]^{+}\ar@{-}[rd]
					&\circ ~3&\\
					& & & & 
					&2~\circ\ar@{-}[r]&\circ\ar@{-}[r]^{+}\ar@{-}[rd]&\circ\ar@{-}[r]_{-}\ar@{-}[ru]
					&\circ\ar@{-}[rr]&&\circ\ar@{-}[r]_{-}\ar@{-}[ru]
					&\circ ~2&\\
					& & & & 
					&1~\circ\ar@{-}[r]&\circ\ar@{-}[r]_{-}\ar@{-}[ru]
					&\circ\ar@{-}[r]&\circ\ar@{-}[rr]&&\circ\ar@{-}[r]
					&\circ ~1&\\   
				}
				\caption*{\hspace{5cm}\mbox{Figure 2: Diagram of $\psi_{\omega}$}}
				\label{fig:psi_omega}
			\end{subfigure}
		}
	\end{figure}
	
	\noindent \textbf{Irreducible Representation of $C(SO_{q}(2n)/SO_q(2n-2))$:}
	Let
	$$\omega_k =\begin{cases}
		I & \text{if } k =1, \\
		s_1s_2\cdots s_{k-1} & \text{if } 2 \leq k \leq n+1,\\
		s_1s_2\cdots s_{n-1}s_{n}s_{n-2}\cdots s_{2n-k} & \text{if } ~n+1 < k <2n,\\
	\end{cases}
	$$
	and $\tilde{\omega}_{n+1}=s_1s_2\cdots s_{n-2}s_n$.  
	Define $\psi_{t,w_k}$ and $\psi_{t,\tilde{\omega}_{n+1}}$ as the restriction of $\pi_{t,w_k}$ and $\pi_{t,\tilde{\omega}_{n+1}}$, respectively, to $C(SO_{q}(2n)/SO_q(2n-2))$. 
	For $k=1$, define $\psi_{t,I}:C(SO_{q}(2n)/SO_q(2n-2))\rightarrow \bbc$ such that $\psi_{t,I}(v_j^{2n}) = t\delta_{1(2n-j+1)}$ for $j\in\{1,2,\hdots,2n\}$. 
	
	\bthm \label{repn of D}
	The collection $\left\{\psi_{t,\omega_k}: t \in \bbbt, 1\leq k < 2n \right\} \cup \left\{\psi_{t,\tilde{\omega}_{n+1}} \right\}$ gives a complete list of irreducible representations of $C(SO_{q}(2n)/SO_q(2n-2))$.
	\ethm
	
	\begin{proof}
		It follows from Theorem \ref{S} and Theorem $2.2$ in \cite{NesTus-2012ab}.
	\end{proof}
	
	Define $\psi_{\omega_k}: C(SO_{q}(2n)/SO_q(2n-2))\rightarrow C(\bbbt) \otimes \scrt^{\otimes \ell({w_k})}$, where $
	\psi_{\omega_k}(a)(t) = \psi_{t,\omega_k}(a)$ for all $a \in  C(SO_{q}(2n+1)/SO_q(2n-1))$.
	\bcrlre
	$\psi_{\omega_{2n-1}}$ is a faithful representation of $C(SO_{q}(2n)/SO_q(2n-2))$.
	\ecrlre

	Similar to the preceding subsection, One can draw the diagram corresponding to the representation $\psi_{w}$; for more details, see \cite[Section 2]{ChaSau-2018ab}. For instance, consider the case where $n=4$ and $w=s_1s_2s_3s_4s_2$. The representation $\psi_{\omega}$ corresponds to Figure 2.

	\section{K-groups of the quotient spaces}
	This section deals with the computation of $K$-groups of the quotient spaces $C(SO_q(2n+1)/SO_q(2n-1))$, $C(SO_q(2n)/SO_q(2n-2))$, and certain intermediate $C^*$-algebras. To understand what is involved, it is instructive to understand the diagram associated with each representation, as given in the last section (see \cite{BhuSau-2023aa} for details).
	
	\subsection{K-groups of $C(SO_q(2n+1)/SO_q(2n-1))$}
	Let \[ 
	B_k^{2n+1}= 
	\begin{cases}
		C(\bbbt) & \mbox{ if } k=1, \cr
		\phi_{\omega_k}(C(SO_q(2n+1)/SO_q(2n-1))) & \mbox{ if } 1<k \leq 2n. \cr
	\end{cases}
	\]
	Observe that $B_k^{2n+1}$ is a $C^*$-subalgebra of $C(\bbbt) \otimes \scrt^{\otimes (k-1)}$.   Moreover, 
	$ \phi_{\omega_k}(v^{2n+1}_{j'})=0$ for $j>k+1$. Therefore, $B_k^{2n+1}$ is generated by $ \phi_{\omega_k}(v^{2n+1}_{j'})$'s with $j\leq k+1$. We denote $\phi_{\omega_k}(v^{2n+1}_{j'})$ by $x_j$ for notational simplicity. Note that for each $k \leq n$, it follows that $x_{k+1} = 0$. Define the homomorphism 
	\[
	\sigma:\scrt \rightarrow \bbc, \,\quad  S \mapsto 1.
	\]
	For $1< k \leq 2n$, let
	$$\rho_k: C(\bbbt) \otimes \scrt^{\otimes (k-1)} \rightarrow C(\bbbt) \otimes \scrt^{\otimes (k-2)}$$ 
	be the homomorphism given by 
	\[
	\rho_k(a)=(1 \otimes 1^{\otimes (k-2)} \otimes \sigma)(a), \quad  \mbox{ for } a \in C(\bbbt) \otimes \scrt^{\otimes (k-1)}.
	\]
	It follows from the description of representations  of $C(SO_q(2n+1)/SO_q(2n-1))$ given in the previous section  (see \cite{BhuSau-2023aa} for details)   that $\rho_k(a) \in B_{k-1}^{2n+1}$ if $a \in B_k^{2n+1}$. This induces a homomorphism 
	$$\rho_k: B_k^{2n+1}\rightarrow B_{k-1}^{2n+1}$$  given by  the restriction of $\rho_k$ to the subalgebra $B_k^{2n+1}$. 
	
	\blmma \label{ideal} 
	Let  $1< k \leq 2n$. Denote by $\mathcal{I}_k$ the kernel of the homomorphism $\rho_k$. Then
	one has the following.
	\[
	I_k=\begin{cases}
		\langle x_{k} \rangle & \mbox{ if } 1< k\leq n, \cr 
		\langle x_{n+1},x_{n+2} \rangle & \mbox{ if } k=n+1, \cr 
		\langle x_{k+1} \rangle & \mbox{ if } n+2 \leq  k \leq 2n. \cr 
	\end{cases}
	\]
	\elmma
	
	\begin{proof}
		We prove the claim for  $1 < k\leq n$. The other cases follows by a similar argument. From the diagram  of $\eta_{\omega_k}$ given in the previous section (see \cite{BhuSau-2023aa}), it is easy to see that $x_k \in \ker \rho_k$, which further implies that $\mathcal{I}_k \subset \ker \rho_k$.  For the converse part,  let $\pi$ be an irreducible representation of $B_k^{2n+1}$ that vanishes on $\mathcal{I}_k$.  It follows from Theorem \ref{repn} that $\pi \cong \pi_{t,w}$, where $w \in W_n$ is a subword of $\omega_k$. Thus,  $\pi$ factors through $\rho_k$.  This proves  that  $\ker \rho_k \subset \mathcal{I}_k$, hence the claim. 
	\end{proof}
	
	\blmma \label{shortexactseq}
	For $1 < k \leq 2n$, one has the following short exact sequence $\chi_{k}$ of $C^*$-algebras. 
	\[
	\chi_{k}: \quad 0\longrightarrow C(\bbbt) \otimes \clk \xrightarrow{i} B_{k}^{2n+1}\xrightarrow{\rho_k}  B_{k-1}^{2n+1}\longrightarrow  0.
	\]
	\elmma

	\prf  
	Fix $1 < k \leq n$. Invoking  Lemma \ref{ideal}, it   is enough to show that $$\mathcal{I}_k=C(\bbbt) \otimes \clk(\ell^2(\bbn))^{\otimes (k-1)}.$$  First, observe that $x_k= t \otimes q^N \otimes \cdots q^N \in  C(\bbbt) \otimes \clk(\ell^2(\bbn_0))^{\otimes (k-1)}$, hence we have 
	$$\mathcal{I}_k \subset C(\bbbt) \otimes \clk(\ell^2(\bbn_0))^{\otimes (k-1)}.$$ 
	Further, note that for $1 \leq l \leq k-1$ we have 
	\[
	x_l=t \otimes (q^N  )^{\otimes (l-1)}\otimes \sqrt{1-q^{2N}}S^* \otimes 1^{\otimes (k-l-1)}.
	\]
	One can verify  that 
	\[
	x_l^jx_k^r1_{\{1\}}(x_k^*x_k)(x_l^*)^i=t^r \otimes p^{\otimes (l-1)}\otimes p_{ij}\otimes p^{\otimes (k-l-1)} \in \mathcal{I}_k
	\]
	for $r,i,j \in \bbn$. By taking product of such elements over $l$, we can see that 
	\[
	t^r \otimes p_{i_1j_1}\otimes p_{i_2j_2}\otimes \cdots \otimes p_{i_{k-1}j_{k-1}}\in \mathcal{I}_k
	\]
	for all $i_1, j_1, i_2, j_2 \cdots i_{k-1},j_{k-1} \in \bbn$. Hence we have 
	\[
	C(\bbbt) \otimes \clk(\ell^2(\bbn_0))^{\otimes (k-1)} \subset \mathcal{I}_k.
	\]
	This settles the case for $1<k \leq n$. In other cases, the claim follows by a similar argument. 
	\qed
	
	\bthm \label{K-groups-B}
	For $1 \leq k \leq 2n$, define $u_k= t \otimes p^{\otimes (k-1)} +1-1\otimes  p^{\otimes (k-1)}$. 
	Then one has 
	\begin{IEEEeqnarray*}{rCl}
		K_0(B_{k}^{2n+1})&=&\begin{cases}
			\langle[1]\rangle \cong \bbz & \mbox{ if } 1 \leq  k \leq n, \cr 
			\langle[1]\rangle \oplus \langle[1 \otimes p^{\otimes (n) }\otimes 1^{\otimes (k-n-1)}]\rangle  \cong \bbz \oplus \bbz/2\bbz & \mbox{ if } n+1 \leq k \leq 2n, \cr 
		\end{cases} \\
		K_1(B_{k}^{2n+1})&=&\langle [u_k] \rangle\cong \bbz.
	\end{IEEEeqnarray*}
	\ethm
	\prf To prove the claim, we will apply induction on $k$. For $k=1$, it is clear. Assume the result to be true for $k-1$. 
	From the Lemma $\ref{shortexactseq}$, 
	we have the short exact sequence,
	\[
	0 \longrightarrow C(\bbbt) \otimes \clk \longrightarrow B_{k}^{2n+1}
	\stackrel{\rho_{k}}{\longrightarrow} B_{k-1}^{2n+1}\longrightarrow 0.
	\]
	which gives rise to the following six-term sequence in $K$-theory.
	\[
	\def\labelstyle{\scriptstyle}
	\xymatrix@C=35pt@R=35pt{
		& K_0(C(\bbbt) \otimes \clk)\ar@{->}[r]
		& K_0(B_{k}^{2n+1})\ar@{->}^{K_{0}(\rho_{k})}[r]
		& K_0(B_{k-1}^{2n+1}))\ar@{->}^{\delta}[d] \\
		&  K_1(B_{k-1}^{2n+1})\ar@{->}^{\partial}[u]\ar@{<-}^{K_{1}(\rho_{k})}[r]
		&K_1(B_{k}^{2n+1})\ar@{<-}[r]       &  K_1(C(\bbbt) \otimes \clk) &
	}
	\]
	To compute the $K$-groups from the above six-term exact sequence, we consider the three following cases separately.\\

	\noindent \textbf{Case (A) $1< k\leq n$:} 
	Since $\rho_{k}(1) = 1$, it follows that $\delta([1]) = 0$.  Further, note that 
	the operator $\widetilde{Y} = t \otimes \underbrace{q^{N} \otimes q^{N} \otimes \cdots \otimes q^N}_{k-2} \otimes S^{*}$ is in $B_{k}^{2n+1}$ as 
	$\eta_{\omega_{k}}(x_{k-1}) - \widetilde{Y}$ lies in $C(\bbbt) \otimes \clk$. Define
	\[
	Y = 1_{\left\{1\right\}}(\widetilde{Y}^{*}\widetilde{Y})\widetilde{Y} +1-1_{\left\{1\right\}}(\widetilde{Y}^{*}\widetilde{Y}).
	\]	
	Then $Y$ is an isometry such that $\rho_{k}(Y) = u_{k-1}$ and hence we have
	\[
	\partial ([u_{k-1}]) = [1-Y^{*}Y]-[1-YY^{*}]= [1 \otimes \underbrace{p \otimes p\otimes ...\otimes p}_{k-1}].
	\]
	Using this,  the claim follows by following the six term exact sequence.\\
	
	\medskip
	\noindent \textbf{Case (B) $k=n+1$:}  In this case, the $\delta$ map is zero as  $\delta([1]) = 0$. However, $\partial$ map is not surjective as in the previous case. To see this, first note that $\widetilde{Z} = t \otimes \underbrace{q^{N} \otimes q^{N} \otimes \cdots }_{n-1} \otimes (S^{*})^2$ is in $B_{n+1}^{2n+1}$ as 
	$\eta_{\omega_{n}}(x_{n}) - \widetilde{Z}$ lies in $C(\bbbt) \otimes \clk$. Define
	\[
	Z = 1_{\left\{1\right\}}(\widetilde{Z}^{*}\widetilde{Z})\widetilde{Z} +1-1_{\left\{1\right\}}(\widetilde{Z}^{*}\widetilde{Z}).
	\]	
	Then $Z$ is an isometry  $ 1 \otimes \underbrace{p \otimes \cdots \otimes p}_{n-1} \otimes (S^{*})^2$ such that $\rho_{n+1}(Y) = u_{n}$ and hence
	\[
	\partial ([u_{n}]) = [1-Z^{*}Z]-[1-ZZ^{*}]= [1 \otimes \underbrace{p \otimes \cdots \otimes p}_{n-1} \otimes (p+p_1)] =2 [1 \otimes \underbrace{p \otimes \cdots \otimes p}_{n}  ] .
	\]  
	Therefore $K_0(i)( [1 \otimes \underbrace{p \otimes \cdots \otimes p}_{n}  ])$ is a nontrivial element and generates the torsion subgroup $\bbz/2\bbz$ of $K_0(C_{n+1})$.   \\

	\medskip
	\noindent \textbf{Case (C) $n+2\leq  k\leq 2n$:}  We assume that  $K_0(B_{k-1}^{2n+1})$ is generated by $[1]$ and $[1 \otimes p^{\otimes (n) }\otimes 1^{\otimes (k-n-2)} ]$, which holds for $k=n+1$ case.  Using the description of representation of $B_{k}^{2n+1}$ (see the diagram),  we have 
	$$\rho_{k}(1) = 1 \,  \mbox{ and }  \,  \rho_{k}(1_{\{1\}}(\eta_{\omega_k}(x_{n+1}x_{n+1}^*)) = \sigma(1 \otimes p^{\otimes (n) }\otimes 1^{\otimes (k-n-1)}) =1 \otimes p^{\otimes (n) }\otimes 1^{\otimes (k-n-2)}.$$
	Hence we get 
	\[
	\delta([1]) = 0 \, \mbox{ and }\,  \delta ([1 \otimes p^{\otimes (n) }\otimes 1^{\otimes (k-n-2)} ]) =0
	\]
	As in the first case, we get surjectivity of $\partial$. The claim now follows by chasing the six term exact sequence. 
	\qed 
	\subsection{K-groups of $C(SO_q(2n)/SO_q(2n-2))$}

	Let \[ 
	D_k^{2n}= 
	\begin{cases}
		C(\bbbt) & \mbox{ if } k=1, \cr
		\psi_{\omega_k}(C(SO_q(2n)/SO_q(2n-2))) & \mbox{ if } 1<k \leq 2n-1. \cr
	\end{cases}
	\]
	Observe that $D_k^{2n}$ is a $C^*$-subalgebra of 
	$C(\bbbt) \otimes \scrt^{\otimes (k-1)}$. 
	Moreover, if $k \leq n$, then 
	\[
	\psi_{\omega_k}(v^{2n}_{j'}) = 0 \quad \text{for } j > k+1,
	\]
	whereas if $k > n$, then 
	\[
	\psi_{\omega_k}(v^{2n}_{j'}) = 0 \quad \text{for } j > k+2.
	\]
	Hence $D_k^{2n}$ is generated by the elements 
	$\psi_{\omega_k}(v^{2n}_{j'})$ with $j \leq k+1$ when $k \leq n$, 
	and with $j \leq k+2$ when $k > n$. 
	For notational convenience, we write
	\[
	y_j := \psi_{\omega_k}(v^{2n}_{j'}).
	\]
	 Let
	$\rho_k$
	be the homomorphism defined in the previous subsection. 
	It follows from the diagram associated with a representation of $D_k^{2n}$ that $\rho_k(a) \in D_{k-1}^{2n}$ if $a \in D_k^{2n}$. This induces a homomorphism 
	from $D_k^{2n}$  to $D_{k-1}^{2n}$  given by  the restriction of $\rho_k$ to the subalgebra $D_{k}^{2n}$. We continue to denote it by the same notation.  
	\blmma \label{ideal1} 
	Let  $1< k \leq 2n-1$. Denote by $\mathcal{J}_k$ the kernel of the homomorphism $\rho_k$. Then
	one has 
	\[
	\mathcal{J}_k=\begin{cases}
		\langle y_{k} \rangle & \mbox{ if } k\leq n, \cr 
		\langle y_{n+1},y_{n+2} \rangle & \mbox{ if } k=n+1, \cr 
		\langle y_{k+1} \rangle & \mbox{ if } n+2\leq k \leq 2n-1, \cr 
	\end{cases}
	\]
	\elmma
	\prf  The claim follows from a similar  computations given in Lemma \ref{ideal}.
	\qed
	\blmma \label{shortexactseq1}
	For $1 < k \leq 2n-1, k\neq n+1$, one has the following short exact sequence $\xi_{k}$ of $C^*$-algebras. 
	\[
	\xi_{k}: \quad 0\longrightarrow C(\bbbt) \otimes \clk \xrightarrow{\iota} D_{k}^{2n}\xrightarrow{\rho_k}  D_{k-1}^{2n}\longrightarrow  0.
	\]
	For $k=n+1$, one has 
	\[
	\xi_{n+1}: \quad 0 \longrightarrow C(SU_q(2)) \otimes \clk \xrightarrow{\iota} D_{n+1}^{2n}\xrightarrow{\rho_{n+1}}  D_{n}^{2n}\longrightarrow  0.
	\]
	\elmma
	\prf We will prove the claim for $k=n+1$. The remaining cases follow along the same line as given in Lemma \ref{shortexactseq}. By Lemma \ref{shortexactseq1}, it remains to show that the ideal  $\mathcal{J}_{n+1}$ generated by $y_{n+1}=t \otimes (q^{N})^{\otimes (n-2)} \otimes \sqrt{1-q^{2N}}S^* \otimes q^N$ and $y_{n+2}=t \otimes (q^{N})^{\otimes n} $ is isomorphic to $C(SU_q(2)) \otimes \clk$.  For that, we will simply interchange the second and $(n-1)$th-tensor component. By doing so, we get   $y_{n+1}= t \otimes \sqrt{1-q^{2N}}S^* \otimes (q^{N})^{\otimes n-1}  \in  C(SU_q(2))  \otimes \clk(\ell^2(\bbn_0))^{\otimes (n-1)}$, hence we have 
	$$\mathcal{J}_k \subset C(SU_q(2))  \otimes \clk(\ell^2(\bbn_0))^{\otimes (n-1)}.$$ 
	Further, note that 
	\[
	y_{n+2}1_{\{1\}}(y_{n+2}^*y_{n+2})=t\otimes p^{\otimes (n)}  \quad  \mbox{ and } \quad 
	y_{n+1}y_{n+1}^*-y_{n+2}^*y_{n+2}=1\otimes 1 \otimes  p^{\otimes (n-2)}.
	\]
	Hence we have 
	\[
	t\otimes p^{\otimes (n)} , 1 \otimes 1  \otimes p^{\otimes (n-1)},   t \otimes \sqrt{1-q^{2N}}S^* \otimes (p)^{\otimes n-1} \in \mathcal{J}_k.
	\]
	This shows that 
	$$ C(SU_q(2)) \otimes  p^{\otimes (n-1)}   \in \mathcal{J}_k.$$
	By applying  $y_1,y_2, \cdots y_{n-1},$ and their adjoints appropriately, one can verify that 
	$$ C(SU_q(2)) \otimes  p_{i_1j_1} \otimes  p_{i_{n-1}j_{n-1}}\in \mathcal{J}_k$$
	for all $i_1, \cdots i_{n-1}, j_1, \cdots j_{n-1}\in \bbn_0$. This proves the claim. 
	\qed\\
	For $1 \leq k \leq 2n-1$, let 
		$$u_k= t \otimes p^{\otimes (k-1)} +1-1\otimes  p^{\otimes (k-1)}.$$
	For $n+1\leq k\leq 2n-1$, define 
	\begin{IEEEeqnarray*}{rCl}
	a_n^k &=& t \otimes p^{\otimes (n-2)} \otimes p \otimes S^* \otimes 1^{\otimes(k-n-1)}, \\
			b_n^k&=& t \otimes p^{\otimes (n-2)}\otimes S^* \otimes p \otimes 1^{\otimes(k-n-1)},\\
			\mathbb{P}_k&=&1 \otimes p^{\otimes (n)} \otimes  1^{\otimes(k-n-1)}, \\
U_k &=&\begin{pmatrix}
	a_n^k+1-  (a_n^k)^*a_n^k  & 	\mathbb{P}_k  \\
		0 &  (b_n^k)^*  +1-  (b_n^k)^*b_n^k
		\end{pmatrix}.
		\end{IEEEeqnarray*}
\bppsn  \label{unitary} 
For $n+1\leq k \leq 2n-1$, the matrix  $U_k$ is a unitary in $M_2(D_k^{2n})$.  Moreover, 
$$\rho_{k}(U_k)=\begin{cases} 
	U_{k-1} & \mbox{ if } k\neq n+1, \cr \\
\begin{pmatrix}
	u_n  & 	0  \\
	0 & 1
\end{pmatrix}& \mbox{ if } k= n+1.\cr\\
\end{cases} $$
 \eppsn 
  \prf Observe that 
 	$$\widetilde{W} = t \otimes \underbrace{q^{N} \otimes q^{N} \otimes \cdots }_{n-1} \otimes S^* \otimes 1^{\otimes (k-n-1)}\in  D_{k}^{2n}
\mbox { as } \psi_{\omega_{k}}(y_{n}) - \widetilde{W} \in \iota(C(SU_q(2)) \otimes \clk).$$ 
Similarly, one has 
$$\widetilde{V} = t \otimes \underbrace{q^{N} \otimes q^{N} \otimes \cdots }_{n-2} \otimes S^* \otimes q^N \otimes 1^{\otimes (k-n-1)} \in  D_{k}^{2n}
\mbox { as } \psi_{\omega_{k}}(y_{n+1}) - \widetilde{V} \in \iota(C(SU_q(2)) \otimes \clk).$$ 
Using this and continuous functional calculus, one can see that $a_n^k, b_n^k\in D_k^{2n}$.  Further,  we have 
$$\mathbb{P}_k=(a_n^k)^*a_n^k-a_n^k(a_n^k)^* \in D_k^{2n} $$ 
The rest follows from a straightforward verification. 

 \qed 
 
	\bthm 
	Let $1\leq k \leq 2n-1$. 
	Then one has 
	\begin{IEEEeqnarray*}{rCl}
		K_0(D_k^{2n})&=&\begin{cases}
			\langle[1]\rangle \cong \bbz & \mbox{ if } 1 \leq  k \leq n, \cr 
			\langle[1]\rangle \oplus \langle[1 \otimes p^{\otimes (n) }\otimes 1^{\otimes (k-n-1)}]\rangle  \cong \bbz \oplus \bbz& \mbox{ if } n+1 \leq k \leq 2n-1, \cr 
		\end{cases} \\
		K_1(D_k^{2n})&=& \begin{cases}
			\langle  [u_k] \rangle \cong \bbz & \mbox{ if } 1 \leq  k \leq n, \cr 
			\langle[u_k]\rangle \oplus \langle[U_k]\rangle  \cong \bbz \oplus \bbz& \mbox{ if } n+1 \leq k \leq 2n-1. \cr 
		\end{cases}
	\end{IEEEeqnarray*}
	\ethm
	\prf The claim is true for $k=1$ as $D_1^{2n}=C(\bbbt)$. Assume the result to be true for $k-1$. To prove the claim for $k$, we split the proof in three cases.

	\noindent \textbf{Case (A) $1< k\leq n$:}  In this case, the proof is similar to the proof of case $(A)$ in Theorem \ref{K-groups-B}.

	\noindent \textbf{Case (B) $k=n+1$:}   
	From the Lemma $\ref{shortexactseq1}$, 
	we have the short exact sequence,
	\[
	\xi_{n+1}: \quad 0\longrightarrow C(SU_q(2)) \otimes \clk \xrightarrow{\iota} D_{n+1}^{2n}\xrightarrow{\rho_{n+1}}  D_{n}^{2n}\longrightarrow  0.
	\]
	which gives rise to the following six-term sequence in $K$-theory.
	\[
	\def\labelstyle{\scriptstyle}
	\xymatrix@C=35pt@R=35pt{
		& K_0(C(SU_q(2)) \otimes \clk)\ar@{->}[r]
		& K_0(D_{n+1}^{2n})\ar@{->}^{K_{0}(\rho_{n+1})}[r]
		& K_0(D_{n}^{2n}))\ar@{->}^{\delta}[d] \\
		&  K_1(D_{n}^{2n})\ar@{->}^{\partial}[u]\ar@{<-}^{K_{1}(\rho_{n+1})}[r]
		&K_1(D_{n+1}^{2n})\ar@{<-}[r]       &  K_1(C(SU_q(2)) \otimes \clk) &
	}
	\]
	Identifying $C(SU_q(2)) \otimes \clk)$ with its image under the injective map $\iota$, we have 
	$$K_0(C(SU_q(2)) \otimes \clk)=\langle [1 \otimes p ^{\otimes (n-2)}\otimes 1 \otimes p]\rangle, \, \, 
	K_1(C(SU_q(2)) \otimes \clk )= \langle [u_{n+1}] \rangle.$$
	Moreover, from the induction hypothesis, we have 
	\[
	K_0(D_{n}^{2n})=	\langle[1]\rangle \cong \bbz, \quad \, \, K_1(D_{n}^{2n})= \langle [u_n] \rangle\cong \bbz.
	\]
	Since $\rho_{n+1}(1)=1$, we get   $\delta([1]) = 0$. Moreover, by Proposition~\ref{unitary}, the element $u_n$ admits a unitary lift $U_{n+1}$ in $D_{n+1}^{2n}$, and hence we get 
		\[
	\partial ([u_{n}]) =[0] .
	\] 
With these facts in hand, the  claim now follows by simply chasing  the six term exact sequence of $K$-groups. 
	
	\noindent \textbf{Case (C) $n+2\leq  k\leq 2n-1$:}  The short exact sequence $\xi_k$ 
	gives rise to the following six-term sequence in $K$-theory.
	\[
	\def\labelstyle{\scriptstyle}
	\xymatrix@C=35pt@R=35pt{
		& K_0(C(\bbbt) \otimes \clk)\ar@{->}[r]
		& K_0(D_{k}^{2n})\ar@{->}^{K_{0}(\rho_{k})}[r]
		& K_0(D_{k-1}^{2n}))\ar@{->}^{\delta}[d] \\
		&  K_1(D_{k-1}^{2n})\ar@{->}^{\partial}[u]\ar@{<-}^{K_{1}(\rho_{k})}[r]
		&K_1(D_{k}^{2n})\ar@{<-}[r]       &  K_1(C(\bbbt) \otimes \clk) &
	}
	\] Using the induction hypothesis,  it follows that $K_0(D_{k-1}^{2n})$ is generated by $[1]$ and $[1 \otimes p^{\otimes (n) }\otimes 1^{\otimes (k-n-2)} ]$.  Using the diagram  associated with  $\psi_{\omega_k}$ of $D_k^{2n}$,  we have 
	$$\rho_{k}(1) = 1 \,  \mbox{ and }  \,  \rho_{k}(1_{\{1\}}(\psi_{\omega_k}(y_{n+1}y_{n+1}^*)) = \rho_{k}(1 \otimes p^{\otimes (n) }\otimes 1^{\otimes (k-n-1)}) =1 \otimes p^{\otimes (n) }\otimes 1^{\otimes (k-n-2)}.$$
	Hence we get 
	\[ 
	\delta([1]) = 0 \, \mbox{ and }\,  \delta ([1 \otimes p^{\otimes (n) }\otimes 1^{\otimes (k-n-2)} ]) =0.
	\]
	Furthermore, by the induction hypothesis, $K_1(D_{k-1}^{2n})$ is generated by $[u_{k-1}]$ and $[U_{k-1}]$. 
	Since $u_{k-1}$ and $U_{k-1}$ admit the unitary lifts $u_k$ and $U_k$, respectively, it follows that
	\[
	\partial([u_{k-1}])=0\,  \mbox{ and }\, \partial([U_{k-1}])=0.
	\] 
	The claim now follows by following the six term exact sequence. 
	\qed 
	
	\newsection{$m$-torsioned quantum double suspension}
	Throughout the section, $A$ will be a unital separable  $C^*$-algebra. For negative $m$, we denote by  $(S^*)^m$ the operator $S^m$.
	
	\bdfn \label{QDS}  Let $A$ be a unital $C^*$-algebra. For $m \in \bbz\setminus \{0\}$, we define its $m$-torsioned quantum double suspension  as the unital $C^*$-algebra $\Sigma_{m}^{2}A$ for which there exists an essential extension 
	\[
	0 \rightarrow A \otimes \clk \rightarrow \Sigma_{m}^{2}A \rightarrow C(\bbbt) \rightarrow 0 
	\]
	such that the corresponding Busby invariant $\beta: C(\bbbt) \rightarrow Q(A \otimes \clk)$ mapping $\beta(t)=[1 \otimes (S^*)^{m}]$. Equivalently, $\Sigma_{m}^{2}A$ can be defined as the $C^*$-subalgebra of $\scrt \otimes A$ generated by $A \otimes \clk$ and $1 \otimes (S^*)^{m}$.  
	\edfn
	\brmrk \label{rm1} \begin{enumerate}[(i)]
		\item Observe that 
		the $C^*$-algebra  $\Sigma_1^{2}A$  is the quantum double suspension $\Sigma^2A$ of $A$ defined in \cite{HonSzy-2002aa}. If $m\neq \pm 1$, the $C^*$-algebra  $\Sigma_m^{2}A$   is a subalgebra of $\Sigma^{2}A$. 
		\item It is clear from the definition that $\Sigma^2_mA\cong \Sigma^2_{-m}A$. 
		\item Let $A$ and $B$ be two unital $C^*$-algebra and $f:A \rightarrow B$ be a unital homomorphism. Define $\tilde{f} :A \rightarrow \Sigma^2B$ by 
		\[
		\tilde{f}(a)= f(a)\otimes p.
		\]
		By the universal property of quantum double suspension, there exists a homomorphism $\bar{f}:\Sigma^2A \rightarrow \Sigma^2B$ such that 
		\[
		\bar{f}(a\otimes k)= f(a)\otimes k, \, k \in \clk \mbox{ and }  \, \bar{f}(1 \otimes S^*)=1 \otimes S^*.
		\]
		Note that $ \bar{f}(\Sigma_m^2A)\subset \Sigma_m^2B$ for all $ m \in \bbn$. This induces a unital homomorphism from $\Sigma_m^2A$ to $\Sigma_m^2B$, which we denote by $\Sigma_m^2f$. With a slight abuse of notation, we denote the same map with an enlarged codomain by the same notation.
	\end{enumerate}
	\ermrk
	
	In what follows, we discuss the representation theory of $m$-torsioned quantum double suspension. In the light of remark (\ref{rm1}), we will assume $m$ to be a positive integer. 
	Let  $\rho$ be an
	irreducible representation of $A$ on a Hilbert space $\clh$. Let $\tilde{\clh}= \clh\otimes \ell^2(\bbn)$ and
	define $\psi_{\rho}:A \rightarrow \cll(\tilde{\clh})$ as $\psi_{\pi}(a)=\pi(a)\otimes p$. From the universal 
	property of $\Sigma^2A$, there exists a representation $\Sigma^2(\rho)$ of $\Sigma^2A$ on $\tilde{\clh}$ 
	such that $\Sigma^2(\rho)\,(1\otimes S^*)=1\otimes S^*$ and  $\Sigma_m^2(\rho)\,(a\otimes p)=\rho(a)\otimes p$ for all $a \in A$. Let $\Sigma_m^2(\rho)$ be the restriction of $\Sigma^2(\rho)$ to $\Sigma_m^2A$.
	\bppsn
	Let $\rho$ be an irreducible representation of $A$. Then $\Sigma_m^2(\rho)$ defined above is an irreducible representation of $\Sigma_m^2A$.
	\eppsn  
	\prf
	Take $\tilde{h} \in \tilde{\clh}$.
	Then there exists an $n \in \bbn_0$ such that $(1\otimes p_n)\tilde{h}\neq0$.
	Hence $h:=(1\otimes p_{n,0})(1\otimes p_n)\tilde{h}\in \clh \otimes e_0$. Since $\pi$ is
	irreducible, one has
	$$\overline{\Sigma_m^2(\rho)(A\otimes p)(\tilde{h})}=\overline{\rho(A)h}\otimes e_0 = \clh \otimes e_0.$$
	Now by applying $\Sigma_m^2(\rho)(1\otimes p_{i,0})$, one can show 
	that $\overline{\Sigma_m^2(\rho)(\Sigma_m^2A)\tilde{h}}=\tilde{\clh}$. This completes the proof.
	\qed 
	
	\bppsn
	Let  $\rho_1$ and $\rho_2$ be two irreducible 
	representations of $A$.  Then $\rho_1$ and $\rho_2$ are unitarily equivalent if and only if $\Sigma_m^2(\rho_1)$ and $\Sigma_m^2(\rho_2)$ 
	of $\Sigma^2A$ are  unitarily equivalent.
	\eppsn
	\prf 
	Assume that $\Sigma_m^2(\rho_1)$ and $\Sigma_m^2(\rho_2)$ are equivalent. 
	Then there exists a unitary operator  $\tilde{U}$ on $\clh \otimes \ell^2(\bbn)$ such
	that $\tilde{U}^*\Sigma^2(\rho_1)(\tilde{a})\tilde{U}=\Sigma^2(\rho_2)(\tilde{a})$ for all  $\tilde{a} \in \Sigma_m^2A$. 
	Hence $\tilde{U}^*(1\otimes p)\tilde{U}=\tilde{U}^*\Sigma^2(\rho_1)(1\otimes p)\tilde{U}=\Sigma^2(\rho_2)(1\otimes p)=1\otimes p$ and 
	similarly  $\tilde{U}(1\otimes p)\tilde{U}^*=1\otimes p$.
	This implies that $\tilde{U}(1\otimes p)$ is a unitary operator on $\clh \otimes e_0$. Let $V : \clh \otimes e_0 \longrightarrow \clh$
	be the unitary operator that sends $h\otimes e_0$ to $h$ for $h \in \clh$. Define $U=V\tilde{U}V^*$. Then it is easy to see that $U$ is 
	an unitary operator on $\clh$ and $U^*\rho_1(a)U=\rho_2(a)$ for all $a \in A$. Therefore $\rho_1$ and $\rho_2$ are equivalent.

	To see the converse part, let $U$ be an  intertwining  unitary  between  $\rho_1$ and $\rho_2$. Then $U \otimes 1$ will be an  intertwining  unitary  between $\Sigma^2_m\rho_1$ and $\Sigma^2_m\rho_2$. 
	\qed 
	
	We will now show that each infinite dimensional irreducible representation of $\Sigma_m^2A$ arises in this manner. 
	Let $\pi$ be an irreducible representation of $\Sigma_m^2A$ on $\tilde{\clh}$. For each $ i \in \bbn_0$, define 
	$H_i=\pi(1\otimes p_i)\tilde{\clh}$.
	Since $\left\{1\otimes p_i\right\}_{i=0}^{\infty}$ are orthogonal projections, $H_i \cap H_j=\left\{0\right\}$ and
	$H_i \bot H_j$ for all $i\neq j \in \bbn_0$. 
	\bppsn \label{chap7-ppsn-invariant}
	Let $\pi$ be a representation of $\Sigma_m^2A$ on $\tilde{\clh}$ and $H_i=\pi(1\otimes p_i)\tilde{\clh}$.
	Then $\oplus_{i=0}^{\infty} H_i$ is an invariant subspace of $\pi$.
	\eppsn
	\prf Take $h \in H_i$. For $i>0$, we have $\pi(a\otimes p)h =  \pi(a\otimes p)\pi(1\otimes p_i)h = 0 \in H_i$. 
	Also,
	\begin{IEEEeqnarray}{rCl}
		\pi(1\otimes S^m)h &=& \pi(1\otimes S^m)\pi(1\otimes p_i)h \nonumber \\
		& = & \pi(1\otimes \left|e_{i-m}\right\rangle \left\langle e_i\right|)h \nonumber \\
		& = &  \pi(1\otimes p_{i-m})\pi(1\otimes \left|e_{i-m}\right\rangle \left\langle e_i\right|)h\nonumber \\
		& \in & H_{i-m}. \nonumber
	\end{IEEEeqnarray}
	
	Similarly,
	one can show that $\pi(1\otimes (S^*)^m)$ keeps $\oplus_{i=0}^{\infty} H_i$  invariant. This proves the claim.
	\qed 
	
	Since $\pi$ is an irreducible representation, either $\oplus_{i=0}^{\infty}H_i$ is $\left\{0\right\}$ or ${\tilde{\clh}}$. \\
	\textbf{Case 1: $\oplus_{i=0}^{\infty}H_i=\tilde{\clh}$}. \\
	Define $T_{ij}: H_i \rightarrow H_j$ as $T_{ij}h=\pi(1\otimes \left|e_j\right\rangle \left\langle e_i\right|)h$.  
	Then it is easy to see that $T_{ij}T_{ji}=id_{H_i}$ and $T_{ji}T_{ij}=id_{H_j}$. This shows that $H_i$ is isomorphic to $H_j$. Hence we have
	$\oplus_{i=0}^{\infty} H_i = H_0 \otimes \ell^2(\bbn)$.
	Define $\rho: A \rightarrow \cll(H_0)$ by
	\begin{IEEEeqnarray}{rCl}\label{chap7-eqn-rho} 
		\rho(a)h=\pi(a\otimes  \left|e_0\right\rangle \left\langle e_0\right|)h. 
	\end{IEEEeqnarray}
	To see that the given map is well defined, take $h \in H_0$. Then
	$\pi(a\otimes \left|e_0\right\rangle \left\langle e_0\right|)h = 
	\pi(1\otimes \left|e_0\right\rangle \left\langle e_0\right|) \pi(a\otimes \left|e_0\right\rangle \left\langle e_0\right|)h \in H_0$. 
	It is easy to see that  $\rho$ is a representation of $A$.  
	\bppsn
	Let $\pi$ be an irreducible representation of $\Sigma_m^2A$. Define 
	$H_i=\pi(1\otimes p_i)\tilde{\clh}, i \in \bbn_0$. If $\oplus_{i=0}^{\infty}H_i=\tilde{\clh}$, then the map  $\rho$ given by equation $(\ref{chap7-eqn-rho})$ 
	is an irreducible  representation of $A$. Moreover, $\pi$ and $\Sigma_m^2(\rho)$ are unitarily equivalent representation of $\Sigma_m^2A$.
	\eppsn
	\prf
	Let $L_0 \subset H_0$ be a proper invariant subspace of $\rho$. Define
	$L_i= \pi(1\otimes \left|e_i\right\rangle \left\langle e_i\right|)L_0$. Now similar calculation as done in 
	Proposition~\ref{chap7-ppsn-invariant}  shows that $\oplus_{i=0}^{\infty} L_i$  is a proper invariant subspace for $\pi$ which 
	contradicts the fact that $\pi$ is irreducible. Hence $\rho$ is an irreducible representation of $A$. Through the  unitary mapping $H_i$ to $H_0\otimes \{e_i\}$ as given above, one can establish the equivalence between $\pi$ and $\Sigma_m^2(\rho)$.
	\qed \\
	\noindent \textbf{Case 2: $\oplus_{i=0}^{\infty}H_i=\left\{0\right\}$}. \\
	Since $H_i = \left\{0\right\}$, one has $ \pi(1\otimes \left|e_i\right\rangle \left\langle e_i\right|)=0$ for all $i \in \bbn$. 
	Also, $\pi(1\otimes \left|e_i\right\rangle \left\langle e_j\right|)=
	\pi(1\otimes \left|e_i\right\rangle \left\langle e_j\right|)\pi(1\otimes \left|e_j\right\rangle \left\langle e_j\right|)=0$ 
	for all $i,j \in \bbn$. This shows  $\pi(1\otimes k)=0$ and hence $\pi(a\otimes k)=0$ for all $k \in \clk$. Therefore
	$\pi(1\otimes S^*)$ is a unitary operator which implies that  image$\left(\pi\right)$ is a commutative $C^*$-algebra. 
	This shows that $\tilde{\clh}=\bbc$  since all irreducible representations of a commutative $C^*$-algebra are one dimensional.
	Hence $\pi(1\otimes (S^*)^m)=t$ for some $t \in \bbbt$.\\
	For $t \in \bbbt$, define 
	\[
	\vartheta_t^m: \Sigma_m^2A \rightarrow \bbc,
	\]
	as $\vartheta_t^m(a\otimes p)=0$ and $\vartheta_t^m(1\otimes (S^*)^m)=t$. This gives all
	one dimensional  inequivalent irreducible representations of $\Sigma_m^2A$.\\ 
	We now summarize above observations in the following theorem.
	\bthm \label{representation of m-torsioned}
	Let $\left\{\rho_i\right\}_{i \in I}$ be all inequivalent irreducible representations 
	of $A$ where $I$ is some indexing set. Then the 
	set $\left\{\Sigma_m^2(\rho_i)\right\}_{i \in I} \cup \left\{\vartheta_t^m\right\}_{t \in \bbbt}$ gives a complete
	list of mutually inequivalent irreducible representations of $\Sigma_m^2A$.
	\ethm
	The following proposition describes a universal property of $m$-torsioned quantum double suspension. 
	\bppsn \label{universal} Let $\phi:A \longrightarrow B$ be a  homomorphism with $\phi(1)=P$. Let $T \in B$ be an isometry with the defect projection $P$ and  $\nu:M_m(\bbc) \rightarrow B$ be a homomorphism satisfying 
	\begin{IEEEeqnarray}{rCl} \label{eq1}
		\nu(1)=P \mbox{ and } \quad \nu(D)\phi(a)=\phi(a)\nu(D),
	\end{IEEEeqnarray} 
	for all $a \in A$, and  $D \in M_m(\bbc)$. Then there exists a unique homomorphism $\Sigma_m^2 (\phi,\nu,T):\Sigma_m^2A \rightarrow B$ such that 
	$$ \Sigma_m^2 (\phi, \nu,T)(a \otimes p_{ij})=\phi(a) \nu(p_{ij}),  \mbox{ for } 0\leq i,j \leq m-1 $$
	and 
	$$ \Sigma_m^2 (\phi,\nu,T)(1 \otimes (S^*)^m)=T.$$
	Conversely, let  $\psi:\Sigma_m^2A \rightarrow B$ be any unital homomorphism. Define $T=\psi(1\otimes (S^*)^m)$ and $P=1-TT^*$. Then there exist homomorphisms 
	$ \phi:A \longrightarrow B$ and $\nu:M_m(\bbc) \rightarrow B$, and an isometry $T$  satisfying equation (\ref{eq1}) such that 
	\[
	\psi =\Sigma_m^2 (\phi, \nu,T).
	\]
	\eppsn
	\prf We will first prove the  forward implication. 
	Observe that the image of $\Sigma_m^2 (\phi, \nu,T)$ is given only for some elements. 
	We will first see that if there is such a homomorphism then there is a unique way to define it to all elements of $\Sigma_m^2A$. For that, take $i,j \in \bbn_0$. Choose $0 \leq i_0, j_0 \leq m-1$ such that $i=k_1m+ i_0$ and $j =k_2m+ j_0$ for some $k_1,k_2\in \bbn_0$. Then 
	$$\Sigma_m^2 (\phi, \nu,T)(a \otimes p_{ij})=\Sigma_m^2 (\phi, \nu,T)((1\otimes (S^*)^m)^{k_1}(a \otimes p_{i_0j_0})(1\otimes S^m)^{k_2})=(T^*)^{k_1}\phi(a)\nu(p_{i_0j_0})T^{k_2}.$$
	Hence one can define $\Sigma_m^2 (\phi, \nu,T)$ on $A\otimes \clk$. Using  $\Sigma_m^2 (\phi, \nu,T)(1\otimes (S^*)^m)=T$, one can extend it to $\Sigma_m^2 A$. To prove that   $\Sigma_m^2 (\phi, \nu,T)$ is a homomorphism,  it is enough to verify that $\Sigma_m^2 (\phi,\nu,T)$ restricted to the subalgebra $1\otimes \clk$ is a homomorphism, which is straightforward to check. For the converse part, define 
	\begin{IEEEeqnarray*}{rCll}
		\phi &:&A \rightarrow B; & \quad \phi(a)=\psi(a \otimes (p_0+p_1+\cdots +p_m)), \mbox{ for all } a \in A;\\
		\nu&:&M_m(\bbc) \rightarrow B; &\quad  \nu(p_{ij})=\psi(1\otimes p_{ij}), \mbox{ for all } 0\leq i,j \leq m-1.
	\end{IEEEeqnarray*}
	It is a not difficult to check  that $\psi =\Sigma_m^2 (\phi, \nu,T)$.
	\qed

	\bcrlre \label{iteration}
	Let $\phi: A \rightarrow Q(\clh)$ be a unital homomorphism. Assume that $\tilde{\phi}: \Sigma^2_mA \rightarrow Q(\clh\otimes \ell^2(\bbn_0))$ is a homomorphism such that 
	\[
	\tilde{\phi}(a \otimes k)= \phi(a)\otimes k,\, k \in \clk(\ell^2(\bbn_0)), \, \mbox{ and } \, \tilde{\phi}(1\otimes (S^*)^m)= T.
	\]
	Then there exists a unique homomorphism  $\tilde{\tilde{\phi}}: \Sigma^2_m\Sigma^2_mA \rightarrow Q(\clh\otimes \ell^2(\bbn_0) \otimes \ell^2(\bbn_0))$  sending 
	$$a \otimes k_1 \otimes k_2\mapsto \phi(a)\otimes k_1 \otimes k_2,\, 1\otimes (S^*)^m\otimes k_3\mapsto T \otimes k_3,\, \mbox{ and }\,  1\otimes 1 \otimes (S^*)^m \mapsto 1\otimes T,$$
	where $k_1,k_2,k_3 \in \clk(\ell^2(\bbn_0))$.
	\ecrlre
	\prf It follows immediately from Proposition \ref{universal}. 
	\qed

	\bthm \label{K-groups for m-QDS}
	Let $K_0(A)$ and $K_1(A)$ be finitely generated abelian groups with
	generators $\left\{\left[P_i\right]\right\}_{i = 1}^{r}$ and  $\left\{\left[U_i\right]\right\}_{i = 1}^s$,  respectively.
	Let $[1] =[P_1]$ be a free generator of $K_0(A)$.  Then
	$K_0(\Sigma_m^2A)$ is isomorphic to $K_0(A)\oplus \bbz/m\bbz$.  
	The  generators $[1]$,$\left\{[P_i\otimes p]\right\}_{i=2}^{r}$  generate the subgroup $K_0(A)$ of $K_0(\Sigma_m^2A)$ and $[1\otimes p]$ generates
	the component $\bbz/m\bbz$ of   $K_0(\Sigma_m^2A)$. Moreover, the group $K_1(\Sigma_m^2A)$ 
	is isomorphic to $K_1(A)$ with generators $\left\{\left[U_i\otimes p+1-1\otimes p\right]\right\}_{i = 1}^s$.
	\ethm
	\prf  The proof follows by simply chasing the six term exact sequence of $K$-groups associated with the short exact sequence 
	\[
	0 \rightarrow A \otimes \clk \rightarrow \Sigma_{m}^{2}A \rightarrow C(\bbbt) \rightarrow 0.
	\]
	\qed

	\newsection{Topological Invariance}
	
	In this section, we prove $q$-invariance  of $SO_q(5)/SO_q(3)$, $SO_q(4)/SO_q(2)$, and  $SO_q(6)/SO_q(4)$. We assume the terminologies related to homogeneous $C^*$-extension theory that are used in  \cite{Sau-2019aa} without any mention. However, we quickly recall some definitions to make the statement of the results accessible to the reader. For a detailed treatment, we refer the reader to \cite{PimPopVoi-1979aa}.  Two homogeneous $C^*$-extensions  $\tau_1$ and $\tau_2$ of $A$ by $Q(C(Y) \otimes \clk)$  are said to be strongly unitarily equivalent   if there exists unitary $U \in M(C(Y) \otimes \clk)$  such that $[U]\tau_1(a)[U^*]=\tau_2(a)$ for all $a \in A$.  We denote it by $\tau_1 \sim_{su} \tau_2$. The strongly unitarily equivalence class of a $C^*$-extension $\tau$ is denoted by $[\tau]_{su}$. 
	
	\subsection{$q$-invariance of $SO_q(3)$}
	
	\bthm \label{SO(3)}
	For $q, q^{\prime} \in (0,1)$, the $C^*$-algebra $C(SO_q(3))$ is isomorphic to  $C(SO_{q^{\prime}}(3))$. Moreover, we have 
	\[
	C(SO_{q^{\prime}}(3)) \cong \Sigma^2_2C(\bbbt),
	\]
	for all $q \in (0,1)$
	\ethm 
	\prf 
	From \cite{Lan-1998aa}, we have the following short exact sequence; 
	\[
	\eta_q: \quad 0\longrightarrow C(\bbbt) \otimes \clk \xrightarrow{\iota} C(SO_q(3)) \rightarrow C(\bbbt) \longrightarrow  0.
	\]
	We denote the corresponding Busby invariant by the same notation $\eta_q$. Note that $\eta_q(\bbt)=[\bbt \otimes (S^*)^2]$. Hence for any $t \in \bbbt$, we have $\mbox{ev}_t\circ\eta_q(\bbt)=[(S^*)^2]$. Since the spectrum of $[(S^*)^2]$ in $Q(\bbbt)$ is $\bbbt$, it follows that $\mbox{ev}_t\circ\eta_q$ is injective for all $t \in \bbbt$. Thus $\eta_q$ is a homogeneous $C^*$-extension. Therefore we have
	$$[\eta_q]_{su}\in \mathrm{Ext}_{\mathrm{PPV}}(\bbbt,C(\bbbt)) \mbox{  for all }  q\in (0,1).$$ 
	Now from Lemma $3.4$ of \cite{Sau-2019aa}, we have 
	\[
	\mathrm{Ext}_{\mathrm{PPV}}(\bbbt,C(\bbbt))=\{[\phi_m]_{su}: m \in \bbz\},
	\]
	and the middle $C^*$-algebra corresponding to the extension $\phi_m$ is $A_m$, which is the same as $\Sigma^2_mC(\bbbt)$ for $m \notin 0$. Moreover, 
	\[
	K_0(A_m)=\begin{cases}
		\bbz\oplus \bbz/m\bbz & \mbox{ for } m\neq 0, \cr
		\bbz\oplus \bbz & \mbox{ for } m=0.\cr
	\end{cases}
	\]
	Therefore for any $q \in (0,1)$ we have  $[\eta_q]_{su} =[\phi_m]_{su}$ for some $m \in \bbz$. 
	Comparing the $K_0$ groups of the middle $C^*$-algebras, one can conclude that 
	$$ \mbox{ either }
	[\eta_q]_{su} =[\phi_{ 2}]_{su} \, \mbox{ or } \,[\eta_q]_{su} =[\phi_{ -2}]_{su} \mbox{ for all } q \in (0,1).$$  Since the middle $C^*$-algebras $A_2$ and $A_{-2}$ of the extensions $[\phi_{ 2}]_{su}$ and $[\phi_{ -2}]_{su}$, respectively, are the same (see \cite{Sau-2019aa}), it follows that 
	$$C(SO_q(3))\cong A_2=\Sigma^2_2C(\bbbt) \mbox{ for all } q \in (0,1).$$
	This proves the claim.
	\qed
	\bcrlre
	Let $t \in \bbbt$ and $w$ be the only nontrivial element of the Weyl group of $\mathfrak{so}(3)$. Then the set $\{ \pi_{t,w}: t \in \bbbt \} \cup \{\tau_t: t \in \bbbt \}$ is the set of all irreducible representations of $C(SO_q(3))$ up to unitarily equivalence.
	\ecrlre
	\prf From Theorem \ref{representation of m-torsioned}, it follows that any irreducible representation of $\Sigma^2_2C(\bbbt)$ is a restriction  of a irreducible representation of $\Sigma^2C(\bbbt)$ to $\Sigma^2_2C(\bbbt)$.  Using this and the fact that $\Sigma^2_2C(\bbbt)\cong C(SU_q(2))$ and $\Sigma^2_2C(\bbbt)\cong C(SO_q(3))$ (by Theorem \ref{SO(3)}), the claim follows.
	\qed 
	\brmrk
	A proof of $q$-independence of $C(SO_q(3))$ is given  in \cite{HonSzy-2002aa}. Here  our attempt is to     fill the gap   in the argument given in \cite{Lan-1998aa}. 
	\ermrk
	
	\subsection{$q$-invariance of $SO_q(5)/SO_q(3)$}
	We introduce the following notational conventions, used henceforth. For \(1 \leq k \leq 2n\), we denote \(B_k^{2n+1}\) to specify its \(q\)-parametrization as \(B_k^{2n+1}(q)\), and the generators \(x_j\) for  \(1 \leq j \leq k+1\) as \(x_{j,q}\). The limit of \(x_{j,q}\) as \(q \to 0\) will be denoted by \(x_{j,0}\). The \(C^*\)-algebra generated by  \(\{x_{j,0}: 1 \leq j \leq k+1\}\) is denoted by \(B_k^{2n+1}(0)\). First, for each \(1 \leq k \leq n+1\), we show the \(q\)-invariance of the intermediate subalgebras \(B_k^{2n+1}(q)\) of \(C(SO_q(2n+1)/SO_q(2n-1))\). Our idea follows the approach of \cite{Sau-2019aa}.

	\blmma \label{lemma-homogeneous}
	For $1 < k \leq 2n$, the short exact sequence 
	\[
	\chi_{k}: \quad 0\longrightarrow C(\bbbt) \otimes \clk \xrightarrow{\iota} B_k^{2n+1}(q) \xrightarrow{\rho_{k}}  B_{k-1}^{2n+1}(q) \longrightarrow  0.
	\]
	is a unital homogeneous extension of $B_{k-1}^{2n+1}(q)$ by 
	$C(\bbbt)\otimes \clk$. 
	\elmma
	\prf Since  $B_{k+1}^{2n+1}(q)$ is unital, the given extension is unital.
	Let $\tau^k: B_{k}^{2n+1}(q) \rightarrow Q(\bbbt)$ be its Busby invariant.  For $t_0 \in \bbbt$, 
	define $\tau_{t_0}^k : B_{k}^{2n+1}(q) \rightarrow Q$ to be $ev_{t_0} \circ \tau^k$.  Assume that 
	$\mathfrak{I}_{t_0}=\ker(\tau_{t_0}^k)$. We need to prove that $\mathfrak{I}_{t_0}=\left\{0\right\}$ for all $t_0 \in \bbbt$. From the diagram of all representations of $C(SO_q(2n+1))$ described in Section $3$ of \cite{BhuSau-2023aa}, we have the following. \\
	\textbf{Case 1: $k\neq n.$} 
	\begin{IEEEeqnarray*}{rCl}
		\tau_{t_0}(y_k^k) &=& t_0[ \underbrace{q^N \otimes \cdots \otimes q^N}_{(k-1)\mbox{ copies }} \otimes \sqrt{1-q^{2N}}S^*]
		\\
		\tau_{t_0}(y_k^k(y_k^k)^*) &=& t_0[ \underbrace{q^{2N} \otimes \cdots \otimes q^{2N}}_{(k-1)\mbox{ copies }} \otimes (1-q^{2N})]=[ \underbrace{q^{2N} \otimes \cdots \otimes q^{2N}}_{(k-1)\mbox{ copies }} \otimes 1]\\
		\tau_{t_0}(y_k^k\mathbbm{1}_{\{y_k^k(y_k^k)^*=1\}}) &=&t_0[\underbrace{p\otimes \cdots \otimes p}_{ (k-1) \mbox{ copies }} \otimes \sqrt{1-q^{2N}}S^*] 
		=t_0[\underbrace{p\otimes \cdots \otimes p}_{ (k-1) \mbox{ copies }} \otimes S^*]. 
	\end{IEEEeqnarray*}
	\textbf{Case 2: $k= n.$} 
	\begin{IEEEeqnarray*}{rCl}
		\tau_{t_0}(y_k^k) &=& t_0[ \underbrace{q^N \otimes \cdots \otimes q^N}_{(k-1)\mbox{ copies }} \otimes \sqrt{1-q^{2N}}(S^2)^*]
		\\
		\tau_{t_0}(y_k^k(y_k^k)^*) &=& t_0[ \underbrace{q^{2N} \otimes \cdots \otimes q^{2N}}_{(k-1)\mbox{ copies }} \otimes (1-q^{2N})]=[ \underbrace{q^{2N} \otimes \cdots \otimes q^{2N}}_{(k-1)\mbox{ copies }} \otimes 1]\\
		\tau_{t_0}(y_k^k\mathbbm{1}_{\{y_k^k(y_k^k)^*=1\}}) &=&t_0[\underbrace{p\otimes \cdots \otimes p}_{ (k-1) \mbox{ copies }} \otimes \sqrt{1-q^{2N}}(S^2)^*] 
		=t_0[\underbrace{p\otimes \cdots \otimes p}_{ (k-1) \mbox{ copies }} \otimes (S^2)^*]. 
	\end{IEEEeqnarray*}
	First, observe that in both cases,  $y_k^k \notin \mathfrak{I}_{t_0}^k$. Thus, the only primitive ideals that contains $\mathfrak{I}_{t_0}^k$ are  maximal ideals $I_t$, and hence 
	$$\mathfrak{I}_{t_0}^k =\cap_{I_t \subset \mathfrak{I}_{t_0}^k } I_t= I_F(\bbbt)\otimes \clk$$ for some closed 
	subset $F$ of $\bbbt$ where $I_F(\bbbt)$ is the closed ideal  of all continuous functions on $\bbbt$ vanishing on $F$. 
	Define the homomorphisms  $$\eta_1: C(\bbbt) \rightarrow Q(\ell^2(\bbn_0)); \quad \bbt \mapsto [S^*],\quad \mbox{ and } \quad \eta_2: C(\bbbt) \rightarrow Q(\ell^2(\bbn_0)); \quad \bbt \mapsto  [(S^*)^2].$$
	Both the  maps $\eta_1$ and $\eta_2$ are injective as the spectrum of $[S^*]$  and $ [(S^2)^*] $ are $\bbbt$.
	Hence  in both cases, we have 
	$$\tau_{t_0}(f(t)\otimes p^{\otimes (k-1)} ) \neq 0$$ 
	for any nonzero function $f$ on $\bbbt$, which further implies that $F=\bbbt$ and $\mathfrak{I}_{t_0}=\left\{0\right\}$. 
	
	\qed

	\blmma 
	Let $q \in (0,1)$ and  $1 \leq k \leq n$. Then one has
	$B_{k}^{2n+1}(q) =B_{k}^{2n+1}(0)$.
	\elmma 
	
	\begin{proof}
		As described in Figure 1, observe that $B_k^{2n+1}(q) = \pi_k\left(C(S_q^{2k+1})\right)$, where $\pi_k$ is the faithful representation of $C(S_q^{2k+1})$ described in \cite{PalSun-2010aa}. Since $B_k^{2n+1}(0) = \pi_k(C(S_0^{2k+1}))$, the lemma follows from \cite[Lemma 3.2]{PalSun-2010aa}.
	\end{proof}
	
	\begin{lmma}
		For all $q\in (0,1)$, we have $B_{n+1}^{2n+1}(q) = B_{n+1}^{2n+1} (0)$.
	\end{lmma}

	\begin{proof}
		For each \(1 \leq j \leq n+2\), we have
		\begin{align*}
			\eta_{\omega_{n+1}}(v_{j'}^{2n+1}) &= \sum_{l=1}^{n-1} (-1)^{l-1} t \otimes (q^N)^{\otimes (l-1)} \otimes S^* \sqrt{1 - q^{2N+2}} \otimes 1^{\otimes (n-l-1)} \otimes \pi_{s_n}(v_{j'}^{l'}) \\
			&\quad + (-1)^{n-1} t \otimes (q^N)^{\otimes (n-1)} \otimes \pi_{s_n}(v_{j'}^{n'}).
		\end{align*}
		
		Therefore, \(B_{n+1}^{2n+1}(q)\) is generated by
		\begin{align*}
			x_{j,q} &= (-1)^{j-1} t \otimes (q^N)^{\otimes (j-1)} \otimes \sqrt{1 - q^{2N}} S^* \otimes 1^{\otimes (n-j)}, \quad 1 \leq j \leq n-1, \\
			x_{n,q} &= (-1)^{n-1} t \otimes (q^N)^{\otimes (n-1)} \otimes \sqrt{(1 - q^{2N})(1 - q^{2(N-1)})} (S^*)^2, \\
			x_{n+1,q} &= (-1)^{n-1} t \otimes (q^N)^{\otimes (n-1)} \otimes \sqrt{(1 + q^2)(1 - q^{2N})} q^{N-1} S^*, \\
			x_{n+2,q} &= (-1)^{n-1} t \otimes (q^N)^{\otimes (n-1)} \otimes q^{2N}.
		\end{align*}
		
		Moreover, \(B_{n+1}^{2n+1}(0)\) is generated by
		\begin{align*}
			x_{j,0} &= (-1)^{j-1} t \otimes p^{\otimes (j-1)} \otimes S^* \otimes 1^{\otimes (n-j)}, \quad 1 \leq j \leq n-1, \\
			x_{n,0} &= (-1)^{n-1} t \otimes p^{\otimes (n-1)} \otimes (S^*)^2, \\
			x_{n+1,0} &= (-1)^{n-1} t \otimes p^{\otimes (n-1)} \otimes p_{10}, \\
			x_{n+2,0} &= (-1)^{n-1} t \otimes p^{\otimes n}.
		\end{align*}
		
		We will prove the lemma by establishing the following claims.
		
		\begin{claim}
			For each $j\in\{1,2,\ldots,n-1\}$, we have $1\otimes p^{\otimes j}\otimes 1^{\otimes (n-j)} \in B_{n+1}^{2n+1}(q) \cap B_{n+1}^{2n+1}(0)$.
		\end{claim}
		
		\begin{claim}
			If \(D\) is a diagonal operator on \(\ell^2(\mathbb{N}_0)\), specified by \(De_i = d(i)e_i\), \(i \in \mathbb{N}_0\) where \(\lim_{i \rightarrow \infty} d(i)\) is either \(0\) or \(1\), then for each \(1 \leq j \leq n\),
			\[ 1 \otimes p^{\otimes (j-1)} \otimes D \otimes 1^{\otimes (n-j)} \in B_{n+1}^{2n+1}(q) \cap B_{n+1}^{2n+1}(0). \]
		\end{claim}
		
		\begin{claim}
			For each $j\in\{1,2,\ldots,n+2\}$, we have $x_{j,0}\in B_{n+1}^{2n+1}(q)$.
		\end{claim}
		
		\begin{claim}
			For each $j\in\{1,2,\ldots,n+2\}$, we have $x_{j,q}\in B_{n+1}^{2n+1}(0)$.
		\end{claim}
		
		\item[\textbf{Proof of Claim 1:}] Note that for each $j\in\{1,2,\ldots, n-1\}$ we have $1\otimes p^{\otimes j}\otimes 1^{\otimes (n-j)} = x_{j+1,0}^* x_{j+1,0}\in B_{n+1}^{2n+1}(0)$. For the other inclusions, it suffices to show for each \(j \in \{1, 2, \ldots, n-1\}\) that
		\[ 1 \otimes (q^{2N})^{\otimes j} \otimes 1^{\otimes (n-j)} \in B_{n+1}^{2n+1}(q), \] which we prove by using backward induction on $j$.
		
		For  $j = n-1$, 
		\begin{align*}
			&x_{n,q}^* x_{n,q} = 1 \otimes (q^{2N})^{\otimes (n-1)} \otimes (1 - q^{2(N+2)})(1 - q^{2(N+1)}) \\
			&\implies 1 \otimes (q^{2N})^{\otimes (n-1)} \otimes 1 - x_{n,q}^* x_{n,q} \in C(\mathbb{T}) \otimes \mathcal{K}^{\otimes n} \subseteq B_{n+1}^{2n+1}(q) \\
			&\implies 1 \otimes (q^{2N})^{\otimes (n-1)} \otimes 1 \in B_{n+1}^{2n+1}(q).
		\end{align*}
		
		Now, if \(1 \otimes (q^{2N})^{\otimes \jmath} \otimes 1^{\otimes (n-\jmath)} \in B_{n+1}^{2n+1}(q)\) for some \(1 < \jmath \leq n-1\), then
		\[
		1 \otimes (q^{2N})^{\otimes (\jmath -1)} \otimes 1^{\otimes (n-\jmath +1)} = q^2 \left(1 \otimes (q^{2N})^{\otimes \jmath} \otimes 1^{\otimes (n-\jmath)} \right) + x_{\jmath,q}^* x_{\jmath,q} \in B_{n+1}^{2n+1}(q),
		\]
		hence the proof follows.
		
		\item[\textbf{Proof of Claim 2:}] It suffices to prove for each \(1 \leq j \leq n\) and \(i \in \mathbb{N}_0\) that 
		\[ 1 \otimes p^{\otimes (j-1)} \otimes p_i \otimes 1^{\otimes (n-j)} \in B_{n+1}^{2n+1}(q) \cap B_{n+1}^{2n+1}(0). \]
		
		Note that
		\[
		1 \otimes p^{\otimes (j-1)} \otimes p_i \otimes 1^{\otimes (n-j)} =
		\begin{cases}
			x_{j+1,0}^* x_{j+1,0} & \text{if } i = 0, \, 1 \leq j \leq n\\
			x_{j,0}^i x_{j+1,0}^* x_{j+1,0} (x_{j,0}^*)^i & \text{if } i \in \mathbb{N}, \, 1 \leq j \leq n-1\\
			x_{n+1,0} x_{n+1,0}^* & \text{if } i = 1, \, j = n\\
			x_{n,0}^{\left\lfloor \frac{i}{2} \right\rfloor} x_{n+1,0} x_{n+1,0}^* (x_{n,0}^*)^{\left\lfloor \frac{i}{2} \right\rfloor} & \text{if } i > 1, \, j = n
		\end{cases},
		\]
		hence for each \(1 \leq j \leq n\) and \(i \in \mathbb{N}_0\) we get \( 1 \otimes p^{\otimes (j-1)} \otimes p_i \otimes 1^{\otimes (n-j)} \in B_{n+1}^{2n+1}(0) \).

		Let \(i \in \mathbb{N}\). Then \(1 \otimes p^{\otimes (n-1)} \otimes p_i \in C(\mathbb{T}) \otimes \mathcal{K}^{\otimes n} \subseteq B_{n+1}^{2n+1}(q)\). Moreover, for \(1 \leq j \leq n-1\),
		\[
		1 \otimes p^{\otimes (j-1)} \otimes p_i \otimes 1^{\otimes (n-j)} = \frac{1}{\prod_{\imath = 1}^i \left(1- q^{2\imath}\right)} \cdot x_{j,q}^i \left(1 \otimes p^{\otimes j} \otimes 1^{\otimes (n-j)} \right) \left(x_{j,q}^* \right)^i \in B_{n+1}^{2n+1}(q).
		\]
		
		\item[\textbf{Proof of Claim 3:}] The proof follows by observing that 
		
		\[ x_{j,0} = \left(1 \otimes p^{\otimes (j-1)} \otimes D_1 \otimes 1^{\otimes (n-j)}\right) x_{j,q} \quad \text{ for } 1\leq j\leq n-1,\]
		\[x_{n,0} =\left(1\otimes p^{\otimes (n-1)}\otimes D_2\right) x_{n,q}, \,x_{n+1,0} = \left(1\otimes p^{\otimes (n-1)}\otimes D_3\right) x_{n+1,q}, \,x_{n+2,0} =\left(1\otimes p^{\otimes n}\right) x_{n+2,q},\]
		
		where the diagonal operators $D_1$, $D_2$ and $D_3$ are defined by
		\[D_1(e_i)=\begin{cases}
			0 &\text{ if } i=0\\
			\frac{1}{\sqrt{1-q^{2i}}}e_i &\text{ if } i\geq 1 
		\end{cases}, \quad D_2(e_i) =\begin{cases}
			0 &\text{ if } i=0,1\\
			\frac{1}{\sqrt{\left(1-q^{2i}\right)\left(1-q^{2(i-1)}\right)}}e_i &\text{ if } i\geq 2
		\end{cases},\]
		\[D_3(e_i) =\begin{cases}
			0 &\text{ if } i\neq 1\\
			\frac{1}{\sqrt{1-q^4}}e_i &\text{ if } i=1
		\end{cases}.\]
		
		\item[\textbf{Proof of Claim 4:} ] Note that 
		$$t\otimes p^{\otimes (n-1)}\otimes p_{(i+1)i}=\begin{cases}
			x_{n+1,0} &\text{ if } i=0\\
			x_{n,0} x_{n+2,0} x_{n+1,0}^* &\text{ if } i=1\\
			x_{n,0}^{\frac{i+1}{2}} x_{n+2,0} x_{n+1,0}^* \left(x_{n,0}^*\right)^{\frac{i-1}{2}} &\text{ if } i \text{ is odd},\,i\geq 3\\
			x_{n,0}^{\frac{i}{2}} x_{n+1,0}\left(x_{n,0}^*\right)^{\frac{i}{2}} &\text{ if } i \text{ is even},\,i\geq 1
		\end{cases},$$
		hence $z_{n+1,q}:= (-1)^{n-1} t\otimes p^{\otimes (n-1)}\otimes q^{N-1} S^* \in B_{n+1}^{2n+1}(0)$.
		
		We now claim for each $1\leq j\leq n+2$ that $y_{j,q}\in B_{n+1}^{2n+1}(0)$, where
		
		$$y_{j,q} := \begin{cases}
			(-1)^{j-1} t\otimes p^{\otimes (j-1)}\otimes \sqrt{1-q^{2N}} S^*\otimes 1^{\otimes (n-j)} &\text{ if } 1\leq j\leq n-1\\
			(-1)^{n-1} t\otimes p^{\otimes (n-1)} \otimes \sqrt{\left(1-q^{2N}\right)\left(1-q^{2(N-1)}\right)} (S^*)^2 &\text{ if } j=n\\
			(-1)^{n-1} t\otimes p^{\otimes (n-1)} \otimes \sqrt{(1+q^2)\left(1-q^{2N}\right)} q^{N-1} S^* &\text{ if } j=n+1\\
			(-1)^{n-1} t\otimes p^{\otimes (n-1)} \otimes q^{2N} &\text{ if } j=n+2
		\end{cases}.$$
		
		Observe that 
		$$y_{j,q} = \begin{cases}
			\left(1\otimes p^{\otimes (j-1)}\otimes \sqrt{1-q^{2N}} \otimes 1^{\otimes (n-j)}\right) x_{j,0} &\text{ if } 1\leq j\leq n-1\\
			\left(1\otimes p^{\otimes (n-1)} \otimes \sqrt{\left(1-q^{2N}\right)\left(1-q^{2(N-1)}\right)}\right) x_{j,0} &\text{ if } j=n\\
			\sqrt{1+q^2}\left(1\otimes p^{\otimes (n-1)}\otimes \sqrt{1-q^{2N}}\right) z_{n+1,q} &\text{ if } j=n+1
		\end{cases},$$
		
		hence for each $1\leq j\leq n+1$ we have $y_{j,q} \in B_{n+1}^{2n+1}(0)$. Moreover,
		$$(-1)^{n-1} t\otimes p^{\otimes (n-1)}\otimes p_i = \begin{cases}
			x_{n+2,0} &\text{ if } i=0\\
			\frac{1}{q^{(i-1)i}} z_{n+1,q}^i x_{n+2,0} \left(z_{n+1,q}^*\right)^i &\text{ if } i\in\mathbb{N}
		\end{cases}$$
		belongs to $B_{n+1}^{2n+1}(0)$, and thus $y_{n+2,q}\in B_{n+1}^{2n+1}(0)$.
		
		Now, we have $x_{1,q}=y_{1,q}\in B_{n+1}^{2n+1}(0)$. For each $2\leq j\leq n-1$, and for each $i_1, i_2,\ldots, i_{j-1}\in\mathbb{N}$ we get
		\begin{align*}
			&(-1)^{j-1} t\otimes p_{i_1}\otimes p_{i_2}\otimes \cdots \otimes p_{i_{j-1}}\otimes \sqrt{1-q^{2N}} S^*\otimes 1^{\otimes (n-j)}\\
			&= \left(x_{1,0}^{i_1} x_{2,0}^{i_2}\cdots x_{j-1,0}^{i_{j-1}}\right) y_{j,q} \left(x_{j-1,0}^*\right)^{i_{j-1}}\left(x_{j-2,0}^*\right)^{i_{j-2}}\cdots \left(x_{1,0}^*\right)^{i_1} \in B_{n+1}^{2n+1}(0),
		\end{align*}
		and this implies that $x_{j,q}\in B_{n+1}^{2n+1}(0)$.
		
		For $j=n$, and for each $i_1, i_2, \ldots, i_{n-1}\in\mathbb{N}$ we get 
		\begin{align*}
			&(-1)^{n-1} t\otimes p_{i_1}\otimes p_{i_2}\otimes \cdots \otimes p_{i_{n-1}}\otimes \sqrt{\left(1-q^{2N}\right)\left(1-q^{2(N-1)}\right)}(S^*)^2\\
			&= \left(x_{1,0}^{i_1} x_{2,0}^{i_2}\cdots x_{n-1,0}^{i_{n-1}}\right) y_{n,q} \left(x_{n-1,0}^*\right)^{i_{n-1}} \left(x_{n-2,0}^*\right)^{i_{n-2}} \cdots \left(x_{1,0}^*\right)^{i_1},
		\end{align*}
		which then implies that $x_{n,q}\in B_{n+1}^{2n+1}(0)$. The proof for $x_{j,q}\in B_{n+1}^{2n+1}(0)$, $j=n+1,n+2$ is similar to the above proof for $j=n$.
	\end{proof}

	We recall that a nuclear, separable, unital $C^*$-algebra $A$ has the homotopy invariance property if for every finite-dimensional metric space $X$ and $[\tau]\in \mathrm{Ext}_{\mathrm{PPV}}(X\times [0,1],A)$, the condition $i_0^*[\tau] = 0$ implies $i_1^*[\tau] = 0$, where $i_t: X \rightarrow X\times [0,1]$ is the injection $i_t(x) = (x,t)$ \cite[Definition 5.6]{PimPopVoi-1979aa}.
	
	\begin{lmma}\label{nuclearity and homotopy invariance}
		For each $q \in (0,1)$ and $1 \leq k \leq 2n$, the $C^*$-algebra $B_k^{2n+1}(q)$ is nuclear and has the homotopy invariance property.
	\end{lmma}
	
	\begin{proof}
		The $C^*$-algebra $B_1^{2n+1}(q) = C(\mathbb{T})$, $C(\mathbb{T}) \otimes \mathcal{K}$ and its unitization are nuclear and quasidiagonal, hence possessing the homotopy invariance property \cite[Proposition 5.5]{PimPopVoi-1979aa}. Using the short exact sequence for $\chi_k$, $1 < k \leq 2n$, we then inductively show that $B_k^{2n+1}(q)$, for $1 \leq k \leq 2n$, are also nuclear and have the homotopy invariance property.
	\end{proof}
	
	\begin{thm}\label{Theorem 5.6}
		For each \(q, q' \in (0,1)\), the \(C^*\)-algebras \(B_{n+2}^{2n+1}(q)\) and \(B_{n+2}^{2n+1}(q')\) are isomorphic.
	\end{thm}
	
	\begin{proof}
		Let \(\tau_q: B_{n+1}^{2n+1}(q) = B_{n+1}^{2n+1}(0) \to Q(\mathbb{T})\) be the Busby invariant corresponding to the short exact sequence \(\chi_{n+1}\) mentioned in Lemma \ref{shortexactseq}. To show that \(B_{n+2}^{2n+1}(q)\) and \(B_{n+2}^{2n+1}(q')\) are isomorphic, it suffices to show that the equivalence classes \([\tau_q]\) and \([\tau_{q'}]\) are equal in \( \mathrm{Ext}_{\mathrm{PPV}}(\mathbb{T}, B_{n+1}^{2n+1}(q)) \).
		
		We claim that \(\tau_q\) and \(\tau_{q'}\) are homotopic. Since \(\tau_t\) is a homomorphism for \(t \in [q, q']\) and considering the topology of point-norm convergence as mentioned in \cite{PimPopVoi-1979aa}, it suffices to show that \(t \mapsto \tau_t\) is continuous. If a sequence \(\{t_m\} \subseteq [q, q']\) converges to \(t\), then for each generator \(x_{j,t} \in B_{n+1}^{2n+1}(t)\), \(1 \leq j \leq n+2\),
		\begin{align*}
			\|\tau_{t_m}(x_{j,t}) - \tau_t(x_{j,t})\| &\leq \|\tau_{t_m}(x_{j,t}) - \tau_{t_m}(x_{j,t_m})\| + \|\tau_{t_m}(x_{j,t_m}) - \tau_t(x_{j,t})\|\\
			&\leq \|x_{j,t} - x_{j,t_m}\| + \|x_{j,t_m} - x_{j,t}\|_{C(\mathbb{T})\otimes \mathscr{T}^{\otimes (n+1)}}
		\end{align*}
		which implies that \(\tau_{t_m}(x_{j,t}) \to \tau_t(x_{j,t})\) in norm. Since \(\|\tau_t\| \leq 1\) for all \(t \in [q, q']\), this implies that \(\tau_{t_m}(P) \to \tau_t(P)\) in norm for any polynomial \(P\) in the generators \(\{x_{j,t}, x_{j,t}^* : 1 \leq j \leq n+2\}\). As these polynomials are dense in \(B_{n+1}^{2n+1}(t)\), it follows that \(\tau_{t_m}(x) \to \tau_t(x)\) in norm for each \(x \in B_{n+1}^{2n+1}(t)\). Hence, \(\tau_q\) and \(\tau_{q'}\) are homotopic.
		
		By Lemma \ref{nuclearity and homotopy invariance}, \(B_{n+1}^{2n+1}(q)\) is a nuclear \(C^*\)-algebra with the homotopy invariance property. Thus, by \cite[Proposition 5.7]{PimPopVoi-1979aa}, \([\tau_q] = [\tau_{q'}]\).
	\end{proof}

	\begin{crlre}\label{$q$-invariance for 5/3}
		For $q, q^{\prime} \in (0,1)$, the $C^*$-algebras $C(SO_q(5)/SO_q(3))$ and $C(SO_{q^{\prime}}(5)/SO_{q^{\prime}}(3))$ are isomorphic.
	\end{crlre}
	
	\begin{proof}
		Observe that the $C^*$-algebra $C\left(SO_q(5)/SO_q(3)\right)$ is isomorphic to $B_4^5(q)$. Therefore, the proof follows by Theorem \ref{Theorem 5.6}.
	\end{proof}
	
	\subsection{$q$-invariance of $SO_q(4)/SO_q(2)$ and $SO_q(6)/SO_q(4)$}
	
	In this subsection, we prove that the $C^*$-algebras $D_k^{2n}(q)$ are $q$-invariant for all $1 \leq k \leq n+2$. This, in particular, ensures the $q$-invariance of $SO_q(4)/SO_q(2)$ and $SO_q(6)/SO_q(4)$.
	
	\blmma \label{lemma-homogeneous1}
	For $1 < k \leq 2n$, the short exact sequence 
	\[
	\xi_{k}: \quad 0\longrightarrow C(\bbbt) \otimes \clk \xrightarrow{\iota} D_k^{2n}(q) \xrightarrow{\rho_{k
	}}  D_{k-1}^{2n}(q) \longrightarrow  0.
	\]
	is a unital homogeneous extension of $D_{k-1}^{2n}(q)$ by 
	$C(\bbbt)\otimes \clk$. 
	\elmma
	\prf Since  $D_{k+1}^{2n}(q)$ is unital, the given extension is unital. Let $\mu^k: D_{k}^{2n}(q) \rightarrow Q(\bbbt)$ be its Busby invariant.  For $t_0 \in \bbbt$, 
	define $\mu_{t_0}^k : D_{k}^{2n}(q) \rightarrow Q$ to be $ev_{t_0} \circ \mu^k$.  Assume that 
	$\mathfrak{J}_{t_0}=\ker(\mu_{t_0}^k)$. To show that $\mathfrak{J}_{t_0}=\left\{0\right\}$ for all $t_0 \in \bbbt$,  we split the proof in two cases. 
	\\
	\textbf{Case 1: $k\neq n+1.$} It follows  by a similar argument given in Lemma \ref{lemma-homogeneous}. \\
	\textbf{Case 2: $k= n+1.$}  From the diagram of representations of $C(SO_q(2n))$ described in \cite{ChaSau-2018ab}, we have the following. \\ 
	\begin{IEEEeqnarray*}{rCl}
		\mu_{t_0}^k(y_k^k) &=& t_0[ \underbrace{q^N \otimes \cdots \otimes q^N}_{(k-1)\mbox{ copies }} \otimes \sqrt{1-q^{2N}}(S^2)^*]
		\\
		\mu_{t_0}^k(y_k^k(y_k^k)^*) &=& t_0[ \underbrace{q^{2N} \otimes \cdots \otimes q^{2N}}_{(k-1)\mbox{ copies }} \otimes (1-q^{2N})]=[ \underbrace{q^{2N} \otimes \cdots \otimes q^{2N}}_{(k-1)\mbox{ copies }} \otimes 1]\\
		\mu_{t_0}^k(y_k^k\mathbbm{1}_{\{y_k^k(y_k^k)^*=1\}}) &=&t_0[\underbrace{p\otimes \cdots \otimes p}_{ (k-1) \mbox{ copies }} \otimes \sqrt{1-q^{2N}}(S^2)^*] 
		=t_0[\underbrace{p\otimes \cdots \otimes p}_{ (k-1) \mbox{ copies }} \otimes (S^2)^*]. 
	\end{IEEEeqnarray*}
	Observe that in both cases,  $y_k^k \notin \mathfrak{J}_{t_0}^k$.  Thus, the only primitive ideals that contains $\mathfrak{J}_{t_0}^k$ are  maximal ideals $I_t$, and hence 
	$$\mathfrak{J}_{t_0}^k =\cap_{I_t \subset \mathfrak{J}_{t_0}^k } I_t= I_F(\bbbt)\otimes \clk$$ for some closed 
	subset $F$ of $\bbbt$ where $I_F(\bbbt)$ is the closed ideal  of all continuous functions on $\bbbt$ vanishing on $F$. 
	Define the homomorphisms  $$\eta_1: C(\bbbt) \rightarrow Q(\ell^2(\bbn_0)); \quad \bbt \mapsto [S^*],\quad \mbox{ and } \quad \eta_2: C(\bbbt) \rightarrow Q(\ell^2(\bbn_0)); \quad \bbt \mapsto  [(S^*)^2].$$
	Both the  maps $\eta_1$ and $\eta_2$ are injective as the spectrum of $[S^*]$  and $ [(S^2)^*] $ are $\bbbt$.
	Hence  in both cases, we have 
	$$\tau_{t_0}(f(t)\otimes p^{\otimes (k-1)} ) \neq 0$$ 
	for any nonzero function $f$ on $\bbbt$, which further implies that $F=\bbbt$ and $\mathfrak{J}_{t_0}=\left\{0\right\}$. 
	\qed

	From now on, for \(1 \leq k \leq 2n-1\), we denote \(D_k^{2n}\) to specify its \(q\)-parametrization as \(D_k^{2n}(q)\), and the generators \(y_j\) for \(1 \leq j \leq k+1\) as \(y_{j,q}\). The limit of \(y_{j,q}\) as \(q \to 0\) will be denoted by \(y_{j,0}\). The \(C^*\)-algebra generated by \(\{y_{j,0}: 1 \leq j \leq k+1\}\) is denoted by \(D_k^{2n}(0)\).
	\blmma 
	Let $q \in (0,1)$ and  $1 \leq k \leq n$. Then one has
	$D_{k}^{2n}(q) =D_{k}^{2n}(0)$.
	\elmma 
	
	\begin{proof}
		It follows from the diagram  of representations of $C(SO_q(2n)/SO_q(2n-2))$ (see \cite{ChaSau-2018ab} and Theorem \ref{repn of D}) that $D_{k}^{2n}(q)$ is same as the $C^*$-algebra underlying the odd dimensional quantum sphere $S_q^{2n+1}$. Hence the claim follows from  \cite[Lemma 3.2]{PalSun-2010aa}. 
	\end{proof}
	
	\begin{lmma}\label{D-1}
		For all $q\in (0,1)$, we have $D_{n+1}^{2n}(q) = D_{n+1}^{2n} (0)$.
	\end{lmma}
	
	\begin{proof}
		We only prove the lemma for $n=2$, and the proof for a general $n$ would follow in similar way. For each $1\leq j\leq n+2$, we have 
		$$\psi_{\omega_{n+1}}\left(v_{j'}^{2n}\right) = t\otimes \sqrt{1-q^{2N}}S^*\otimes \pi_{s_2}\left(v_{j'}^{2n}\right) - t\otimes q^N\otimes \pi_{s_2}\left(v_{j'}^{2n-1}\right).$$
		Therefore $D_{n+1}^{2n}(q)$ is generated by 
		\begin{equation*}
			\begin{aligned}[c] 
				y_{1,q} &= t\otimes \sqrt{1-q^{2N}}S^* \otimes \sqrt{1-q^{2N}}S^*,\\
				y_{3,q} &= -t\otimes \sqrt{1-q^{2N}}S^*\otimes q^N,
			\end{aligned}
			\quad
			\begin{aligned}[c]
				y_{2,q} &= -t\otimes q^N\otimes \sqrt{1-q^{2N}}S^*,\\
				y_{4,q} &= -t\otimes q^N\otimes q^N.
			\end{aligned}
		\end{equation*}
		Moreover, $D_{n+1}^{2n}(0)$ is generated by 
		$$y_{1,0}= t\otimes S^*\otimes S^*,\, y_{2,0}= -t\otimes p\otimes S^*,\, y_{3,0}= -t\otimes S^*\otimes p, \text{ and } y_{4,0}= -t\otimes p\otimes p.$$
		
		We will prove the lemma by establishing the following claims.

		\setcounter{claim}{0}
		\begin{claim}
			$\{1\otimes p\otimes 1,\, 1\otimes 1\otimes p,\, 1\otimes p\otimes p\} \subset D_{n+1}^{2n}(q)\cap D_{n+1}^{2n}(0)$.
		\end{claim}
		
		\begin{claim}
			If $D$ is a diagonal operator on $\ell^2(\mathbb{N}_0)$, specified by $De_i = d(i)e_i$, $i\in\mathbb{N}_0$ where $\lim_{i\rightarrow\infty} d(i)$ is either $0$ or $1$, then 
			$$\{1\otimes D\otimes D,\, 1\otimes p\otimes D,\, 1\otimes D\otimes p\} \subset D_{n+1}^{2n}(q) \cap D_{n+1}^{2n}(0).$$
		\end{claim}
		
		\begin{claim}
			For each $j\in\{1,2,\ldots, n+2\}$, we have $y_{j,0}\in D_{n+1}^{2n}(q)$.
		\end{claim}
		
		\begin{claim}
			For each $j\in\{1,2,\ldots, n+2\}$, we have $y_{j,q}\in D_{n+1}^{2n}(0)$.
		\end{claim}
		
		\item[\textbf{Proof of Claim 1:}] First note that 
		$$\{1\otimes p\otimes 1, 1\otimes 1\otimes p, 1\otimes p\otimes p\} = \{y_{j+1,0}^* y_{j+1,0} : j=1,2,3\} \subset D_{n+1}^{2n}(0).$$
		
		For other inclusions, it suffices to show that
		\[
		\{1\otimes q^{2N} \otimes 1,\, 1\otimes 1\otimes q^{2N},\, 1\otimes q^{2N}\otimes q^{2N}\} \subset D_{n+1}^{2n}(q),
		\]
		which follows from the relations:
		\[
		1\otimes q^{2N}\otimes 1 = y_{2,q}^* y_{2,q} - q^2\cdot y_{4,q}^* y_{4,q}, \, 1\otimes 1\otimes q^{2N} = y_{3,q}^* y_{3,q} - q^2\cdot y_{4,q}^* y_{4,q}, \text{ and }
		1\otimes q^{2N}\otimes q^{2N} = y_{4,q}^* y_{4,q}.
		\]
		
		\item[\textbf{Proof of Claim 2:}] Using Claim 1, it suffices to prove for each $i,j\in\mathbb{N}_0$ that 
		$$1\otimes p_i\otimes p_j \in D_{n+1}^{2n}(q) \cap D_{n+1}^{2n}(0).$$
		Since $C(\mathbb{T})\otimes\mathcal{K}\otimes\mathcal{K} \subseteq D_{n+1}^{2n}(q)$, we have $1\otimes p_i\otimes p_j \in D_{n+1}^{2n}(q)$. To ensure the other belonging note that
		
		$$1\otimes p_i\otimes p_j = \begin{cases}
			y_{1,0}^i y_{2,0}^{j-i} (1\otimes p\otimes p) (y_{2,0}^*)^{j-i} (y_{1,0}^*)^i &\text{ if } i\leq j\\
			y_{1,0}^j y_{3,0}^{i-j}(1\otimes p\otimes p) (y_{3,0}^*)^{i-j}(y_{1,0}^*)^j &\text{ if } i\geq j
		\end{cases}.$$
		
		\item[\textbf{Proof of Claim 3:}] Let $D$ be the diagonal operator on $\ell^2(\mathbb{N}_0)$, specified by $$De_i=\begin{cases}
			0 &\text{ if } i=0\\
			\frac{1}{\sqrt{1-q^{2i}}}e_i &\text{ if } i\in\mathbb{N}
		\end{cases}.$$ The proof then follows by observing that 
		\begin{equation*}
			\begin{aligned}[c]
				y_{1,0} &= (1\otimes D\otimes D)y_{1,q},\\
				y_{3,0} &= (1\otimes D\otimes p)y_{3,q},
			\end{aligned}
			\quad
			\begin{aligned}[c]
				y_{2,0} &= (1\otimes p\otimes D)y_{2,q},\\
				y_{4,0} &= (1\otimes p\otimes p)y_{4,q}.
			\end{aligned}
		\end{equation*}
		
		\item[\textbf{Proof of Claim 4:}] By Claim 2, we obtain
		\[
		y_{1,q} = \left(1 \otimes \sqrt{1-q^{2N}} \otimes \sqrt{1-q^{2N}}\right) y_{1,0} \in D_{n+1}^{2n}(0).
		\]
		To show that \( y_{2,q} \in D_{n+1}^{2n}(0) \), it suffices, by the binomial expansion of \( \sqrt{1 - q^{2N}} \), to prove for each \( l \in \mathbb{N}_0 \) that $t \otimes q^N \otimes q^{2lN} S^* \in D_{n+1}^{2n}(0)$. We first establish that \( t \otimes q^{\frac{N}{2}} \otimes S^* \in D_{n+1}^{2n}(0) \). This is sufficient if we can show for each \( i \in \mathbb{N}_0 \) that $t \otimes p_i \otimes S^* \in D_{n+1}^{2n}(0)$.
		Observe that \( t \otimes p \otimes S^* = -y_{2,0} \in D_{n+1}^{2n}(0) \). For \( i \in \mathbb{N} \), we have
		\[
		t \otimes p_i \otimes S^* = -y_{1,0}^i y_{2,0} \left( y_{1,0}^* \right)^i + \sum_{r=0}^{i-1} y_{1,0}^{r+1} y_{3,0}^{i-r-1} y_{4,0} \left( y_{3,0}^* \right)^{i-r} \left( y_{1,0}^* \right)^r \in D_{n+1}^{2n}(0).
		\]
		Hence, \( t \otimes q^{\frac{N}{2}} \otimes S^* \in D_{n+1}^{2n}(0) \), as required. Using the proof of Claim 2, it follows that for each \( l \in \mathbb{N}_0 \),
		$1 \otimes q^{\frac{N}{2}} \otimes q^{2lN} \in D_{n+1}^{2n}(0)$, and thus,
		\[
		t \otimes q^N \otimes q^{2lN} S^* = \left( 1 \otimes q^{\frac{N}{2}} \otimes q^{2lN} \right) \left( t \otimes q^{\frac{N}{2}} \otimes S^* \right) \in D_{n+1}^{2n}(0).
		\]
		
		Similarly, to show \( y_{3,q} \in D_{n+1}^{2n}(0) \), it suffices to show for each \( l \in \mathbb{N}_0 \) that 
		$t \otimes q^{2lN} S^* \otimes q^N \in D_{n+1}^{2n}(0)$,
		which follows from showing $t \otimes S^* \otimes q^{\frac{N}{2}} \in D_{n+1}^{2n}(0)$. Observe that \( t \otimes S^* \otimes p = -y_{3,0} \in D_{n+1}^{2n}(0) \). For \( i \in \mathbb{N} \),
		\[
		t \otimes S^* \otimes p_i = -y_{1,0}^i y_{3,0} \left(y_{1,0}^*\right)^i + \sum_{r=0}^{i-1} y_{1,0}^{r+1} y_{2,0}^{i-r-1} y_{4,0} \left(y_{2,0}^*\right)^{i-r} \left(y_{1,0}^*\right)^r \in D_{n+1}^{2n}(0),
		\] 
		hence \( t \otimes S^* \otimes q^{\frac{N}{2}} \in D_{n+1}^{2n}(0) \).
		
		Finally, for \( i,j \in \mathbb{N}_0 \),
		\[
		t \otimes p_i \otimes p_j = \begin{cases}
			-y_{1,0}^i y_{2,0}^{j-i} y_{4,0}\left(y_{2,0}^*\right)^{j-i}\left(y_{1,0}^*\right)^i &\text{ if } i \leq j\\
			-y_{1,0}^j y_{3,0}^{i-j} y_{4,0} \left(y_{3,0}^*\right)^{i-j}\left(y_{1,0}^*\right)^j &\text{ if } i \geq j
		\end{cases}.
		\]
		Thus, \( t \otimes p_i \otimes p_j \in D_{n+1}^{2n}(0) \), implying \( y_{4,q} \in D_{n+1}^{2n}(0) \).
		
	\end{proof}
	
	\begin{crlre}
		The quotient space $SO_q(4)/SO_q(2)$ is $q$-invariant.
	\end{crlre}
	
	\begin{proof}
		Observe that the $C^*$-algebra $C\left(SO_q(4)/SO_q(2)\right)$ is isomorphic to $D_3^4(q)$. Therefore, by Lemma \ref{D-1}, the quotient space $SO_q(4)/SO_q(2)$ is $q$-invariant.
	\end{proof}
	
	\begin{lmma}\label{nuclearity and homotopy invariance1}
		For each $q \in (0,1)$ and $1 \leq k \leq 2n-1$, the $C^*$-algebra $D_k^{2n}(q)$ is nuclear and has the homotopy invariance property.
	\end{lmma}
	\prf The proof is similar to that of Lemma \ref{nuclearity and homotopy invariance}, hence omitted.
	\qed
	\begin{thm}
		
		For $q, q^{\prime} \in (0,1)$, the \(C^*\)-algebras \(D_{n+2}^{2n}(q)\) and \(D_{n+2}^{2n}(q')\) are isomorphic. As a consequence, 
		the $C^*$-algebra $C(SO_q(6)/SO_q(4))$ is isomorphic to $C(SO_{q^{\prime}}(6)/SO_{q^{\prime}}(4))$ for any $q, q^{\prime} \in (0,1)$.
	\end{thm}
	
	\prf We will only outline the proof as all the calculations involved are similar to the case of $B_{n+2}^{2n+1}(q)$. 
	\begin{enumerate}[(i)]
		\item First,  the homomorphisms $\mu_q^k$ and $\mu_{q^{\prime}}^k$ are homotopic. 
		\item Moreover, by Theorem \ref{D-1} and Lemma \ref{lemma-homogeneous1}, one has $[\mu_q^k]_{su} , \, [\mu_{q^{\prime}}^k]_{su} \in \mathrm{Ext}_{\mathrm{PPV}}(\bbbt, D_{n+1}^{2n}(0))$ .
		\item By Lemma \ref{nuclearity and homotopy invariance1}, $D_{n+1}^{2n}(0)$ is nuclear and has the homotopy invariance property. Therefore, by \cite[Proposition 5.7]{PimPopVoi-1979aa}, we have   $[\mu_q^k]_{su} =  [\mu_{q^{\prime}}^k]_{su}$.
		\item Hence, the middle $C^*$-algebras $D_{n+2}^{2n}(q)$ and $D_{n+2}^{2n}(q^{\prime})$ of the extensions $\mu_q^k$ and $\mu_{q^{\prime}}^k$, respectively,  are isomorphic. 	In particular, by taking \(n=3\),  we get an isomorphism between  $C(SO_q(6)/SO_q(4))$ and  $C(SO_{q^{\prime}}(6)/SO_{q^{\prime}}(4))$.
	\end{enumerate} 
	\qed 
	
	\noindent\begin{footnotesize}\textbf{Acknowledgement}:
		This work was carried out while the first and second named authors were at IIT Gandhinagar, and they gratefully acknowledge the support provided by the institute. The second named author also gratefully acknowledges support from the NBHM research project grant 02011/19/2021/NBHM(R.P)/R\&D II/8758, and INSPIRE Faculty Fellowship research grant DST/INSPIRE/04/2021/002620 (DST, Govt. of India). The third named author gratefully acknowledges support from NBHM grant 02011/19/2021/NBHM(R.P)/R\&D II/8758.
	\end{footnotesize}
	
	\noindent \textbf{Data Availability:} No data is used in the research described in this paper.
	
	\section*{Declarations}
	\noindent \textbf{Conflict of interest:} 
	The authors declares that they have no conflict of interest.

	\bigskip

	\noindent{\sc Akshay Bhuva} (\texttt{akshaybhuva98@gmail.com})\\
	{\footnotesize Department of Mathematics,\\ Nirma University, Institute of Technology,\\  Ahmedabad 382481, Gujarat, India}\\
	
	\noindent{\sc Surajit Biswas} (\texttt{surajit.b@srmap.edu.in},  \texttt{241surajit@gmail.com})\\
	{\footnotesize Department of Mathematics,\\ SRM University AP, Neerukonda,\\ Mangalagiri Mandal, Guntur District,\\ Andhra Pradesh 522240, India}\\
	
	\noindent{\sc Bipul Saurabh} (\texttt{bipul.saurabh@iitgn.ac.in},  \texttt{saurabhbipul2@gmail.com})\\
	{\footnotesize Department of Mathematics,\\ Indian Institute of Technology, Gandhinagar,\\  Palaj, Gandhinagar 382055, India}

\end{document}